\documentclass[a4,11pt]{article}

\usepackage{amsmath,amssymb,amscd,amsthm}
\usepackage[all]{xy}
\usepackage{enumerate}
\usepackage[dvips]{graphicx}
\usepackage{wrapfig}

\theoremstyle{plain}
\numberwithin{equation}{section}
\newtheorem{thm}{Theorem}[section]
\newtheorem{prop}[thm]{Proposition}
\newtheorem{cor}[thm]{Corollary}
\newtheorem{lem}[thm]{Lemma}

\theoremstyle{definition}
\newtheorem{dfn}[thm]{Definition}
\newtheorem{ex}[thm]{Example}
\newtheorem{rmk}[thm]{Remark}

\def\rank{\mathop{\mathrm{rank}}\nolimits}
\def\dim{\mathop{\mathrm{dim}}\nolimits}
\def\Im{\mathop{\mathrm{Im}}\nolimits}

\def\Hom{\mathop{\mathrm{Hom}}\nolimits}
\def\hom{\mathop{\mathrm{hom}}\nolimits}
\def\Ext{\mathop{\mathrm{Ext}}\nolimits}

\def\<{{\langle}}
\def\>{{\rangle}}

\def\Aut{\mathop{\mathrm{Aut}}\nolimits}
\def\Stab{\mathop{\mathrm{Stab}}\nolimits}

\def\+{\mathop{\oplus}\nolimits}

\newcommand{\mf}[1]{{\mathfrak{#1}}}

\newcommand{\bb}[1]{{\mathbb{#1}}}
\newcommand{\mca}[1]{{\mathcal{#1}}}
\newcommand{\mr}[1]{{\mathrm{#1}}}

\title{Stability of Gieseker stable sheaves on K3 surfaces in the sense of Bridgeland and some applications}
\author{Kotaro Kawatani\footnote{2010 Mathematics Subject Classification. Primary 14F05, 14J28, Secondary 14J60, 14J10}}
\date{}
 
\begin{document}
\maketitle

\begin{abstract}
We show that some Gieseker stable sheaves on a projective K3 surface $X$ are stable with respect to a stability condition of Bridgeland on the derived category of $X$ if the stability condition is in explicit subsets of the space of stability conditions depending on the sheaves. 
Furthermore we shall give two applications of the result. 
As a part of these applications, we show that the fine moduli space of Gieseker stable torsion free sheaves on a K3 surface with Picard number one is the moduli space of $\mu$-stable locally free sheaves if the rank of the sheaves is not a square number. 
\end{abstract}

\tableofcontents

\section{Introduction} 

%
In the article $\cite{Bri}$, Bridgeland constructed the theory of stability conditions on arbitrary triangulated categories $\mca D$. 
A stability condition $\sigma$ is a pair $(\mca A, Z)$ with some axioms where $\mca A$ is the heart of a bounded $t$-structure on $\mca D$ and $Z$ is a group homomorphism from the Grothendieck group $K(\mca D)$ of $\mca D$ to $\bb C$. 
Let $\Stab (\mca D)$ be the space of stability conditions on $\mca D$. 
If $\Stab (\mca D)$ is not empty then $\Stab (\mca D)$ is known to be a complex manifold by \cite{Bri}. 
If a stability condition $\sigma$ on $\mca D$ exists we can define the notion of $\sigma$-stability for objects $E \in \mca D$.

Suppose that $\mca D$ is the bounded derived category $D(X)$ of coherent sheaves on a  smooth projective variety $X$ over $\bb C$. 
In this paper we study the case where $X$ is a projective K3 surface. 
Then as is well-known, the space $\Stab (X)$ of stability conditions on $D(X)$ is not empty by virtue of Bridgeland \cite{Bri2}. 
Then for coherent sheaves on $X$ we have have the notion of $\sigma$-stability in addition to Gieseker stability and $\mu$-stability. 
Thus it is natural to compare these stabilities. 
We shall give an partial answer to this problem.

We have two goals. 
The first goal is to show the $\sigma$-stability of Gieseker stable (or $\mu$-stable) sheaves on $X$ if $\sigma$ is in explicit subsets of $\Stab (X)$ depending on the sheaves. 
This result will be proved in Theorems \ref{A.4} and \ref{A.10}. 
The second goal is to give two applications of these two theorems.

We comment on Theorems \ref{A.4} and \ref{A.10}. 
Recall that the space $\Stab (X)$ has the subset $U(X)$ described by
\begin{eqnarray}
U(X) &=& \{ \sigma  \in \Stab (X)|\forall x \in X,\  \mca O_x\mbox{ is $\sigma$-stable with a common phase} \notag \\
&&\mbox{ and $\sigma$ is good, locally finite and numerical}   \},
\label{U(X)}
\end{eqnarray}
where $\mca O_x$ is the structure sheaf of a closed point $x \in X$. 
Very roughly this subset $U(X)$ is also a trivial $\tilde {GL}^+(2, \bb R)$ bundle over a set $V(X)$, where $\tilde {GL}^+(2, \bb R)$ is the universal cover of $GL^+(2, \bb R)$ (See also Section $3$). 
In addition $V(X)$ is roughly parametrized by $\bb R$ divisors $\beta$ and $\bb R$ ample divisors $\omega$.   
Hence we can write $\sigma \in U(X)$ as $\sigma = \sigma_{(\beta, \omega)} \cdot \tilde  g$ where $\sigma _{(\beta,\omega)}\in V(X)$ and $\tilde g \in \tilde{GL}^+(2, \bb R)$. 
 It is shown in \cite{Bri2} that if we take a sufficiently large $\omega >> 0$, then the notion of $\sigma$-stability is just $(\beta, \omega)$-twisted stability for coherent sheaves. 
Namely, for any sufficiently large $\lambda >>0$ if $E \in D(X)$ is $\sigma_{(\beta, \lambda \omega)}\cdot \tilde g$-stable then $E$ is a $(\beta, \omega )$-twisted stable sheaf and vice versa. 
In some sense we strengthen this result. 
We give an explicit bound for $\lambda$ depending on sheaves so that Gieseker stable sheaf is $\sigma_{(\beta,\lambda \omega)}\cdot \tilde g$-stable.

In Theorem \ref{B.4}, which is the first application, we study fine moduli spaces of Gieseker stable sheaves on a projective K3 surface $X$ with Picard number one. 
More precisely we show the fine moduli space of Gieseker stable torsion free sheaves is the fine moduli space of $\mu$-stable locally free sheaves if the rank of the sheaves is not a square number. 
We also show that if the rank is a square number then the fine moduli space is the moduli space of $\mu$-stable locally free sheaves or the moduli space of properly Gieseker stable torsion free sheaves\footnote{Namely the sheaf is neither $\mu$-stable nor locally free. }. 
Furthermore we show that if the latter case occurs then the moduli space is isomorphic to $X$ itself. 
The key idea of the proof of Theorem \ref{B.4} is to compare two Jordan-H\"{o}lder filtrations of a Gieseker stable sheaf with respect to $\mu$-stability and $\sigma$-stability for some $\sigma \in \Stab (X)$. 
This comparison is enabled by Proposition \ref{B.2}. 
In this proposition we show that the $\sigma$-stability of some Gieseker stable sheaf $E$ on $X$ is equivalent to the $\mu$ stability and the local freeness of $E$ if $\sigma$ is in a subset $V^-_{v(E)}$ (See Section 5 for the definition of $V_{v(E)}^-$). 
As a consequence of Theorem \ref{B.4}, we see that any non trivial Fourier-Mukai partner should be the fine moduli of $\mu$-stable locally free sheaf. 

The second application is Theorem \ref{C.7}, which is a generalization of \cite[Theorem 1.1]{Kaw}. 
Let $\Phi:D(Y) \to D(X)$ be an equivalence where $X$ and $Y$ are projective K3 surfaces with Picard number one and let $\Phi_* :\Stab (Y) \to \Stab (X)$ be a natural map induced by $\Phi$. 
In \cite[Theorem 1.1]{Kaw}, 
the author showed that, if $\Phi$ satisfies the condition $\Phi_* U(Y) = U(X)$ then the equivalence $\Phi$ is given by $M \otimes f_* (-)[n]$ 
where $M$ is a line bundle on $X$, $f$ is an isomorphism from $Y$ to $X$ and $n \in \bb Z$.

As the second application, we remove the assumption that the Picard number of $X$ is one from \cite[Theorem 1.1]{Kaw}. 
We proceed as follows. 
For an equivalence $\Phi :D(Y) \to D(X)$ satisfying the assumption $\Phi_* U(Y)=U(X)$, 
one can see that it is enough to prove $\Phi(\mca O_y) =\mca O_x[n]$ where $x \in X$ and $n \in \bb Z$, since $U(X)$ is given by (\ref{U(X)}). 
In \cite{Kaw}, this was proved by using \cite[Theorem 6.6]{Kaw}. 
Hence the crucial part of the proof of \cite[Theorem 1.1]{Kaw} is \cite[Theorem 6.6]{Kaw}. 
A necessary generalization of this result of \cite{Kaw} will be done in Corollary \ref{C.6} by applying Theorem \ref{A.6}. 

We finally explain the motivation of our study.  
In the previous paper \cite {Kaw} we also showed the $\sigma$-stability of Gieseker stable or $\mu$-stable sheaves. 
Before we started the previous study we expected that there would be a Gieseker stable torsion free sheaf $E$ with $\dim \Ext^1_X(E,E)=2$ on a polarized K3 surface $(X, L)$ such that $E$ is $\sigma$-stable for all $\sigma$ in $U(X)$. 
This conjecture is based on the fact that any line bundles $P$ with $c_1(P)=0$ on an abelian surface\footnote{These line bundles are Gieseker stable with $\dim \Ext^1_X(E, E) =2$. } are $\sigma$-stable for all $\sigma \in U(X)$. 
However throughout the previous study we showed our conjecture never holds if $X$ is a projective K3 surface with Picard number one. 
Hence we had to give up our first conjecture and tried the following two things. 
One is to find explicit subsets of $U(X)$ depending on Gieseker stable sheaves so that the sheaves are $\sigma$-stable if $\sigma$ is in the subsets. 
The other is to find interesting applications of $\sigma$-stability of Gieseker stable sheaves.

\section{Review of classical stability for sheaves}

In this section we recall the $\mu$-stability, Gieseker stability and twisted stability for coherent sheaves on a projective K3 surface. 

We first introduce some notations. 
Throughout this section $X$ is a projective K3 surface over $\bb C$. 
Let $A$ and  $B$ be in $D(X)$. 
If the $i$-th cohomology $H^i(A)$ is concentrated only at degree $i=0$, we call $A$ a \textit{sheaf}. 
We put $\Hom ^n_X(A,B) := \Hom_{D(X)}(A,B[n]) $ and $\hom^n_{X}(A,B) := \dim _{\bb C}\Hom_X^n(A,B)$ where $[n] $ means $n \in \bb Z $ times shifts. 
We remark that 
\[
\Hom_X^n(A,B) = \Hom^{2-n}_X(B,A)^*
\]
by the Serre duality. 
Then the Euler paring $\chi(E, F)= \sum_{i}(-1)^i \hom_X^i(E,F)$ is a $\bb Z$-bilinear symmetric form on the Grothendieck group $K(X)$ of $D(X)$. 

Let $\mca N(X)$ be the quotient of $K(X)$ by numerical equivalence with respect to the Euler pairing $\chi$. 
Then $\mca N(X)$ is isomorphic to $$H^0(X, \bb Z) \+ \mr{NS}(X)\+ H^4(X, \bb Z)$$ where $\mr{NS}(X)$ is the N\'{e}ron-Severi Lattice of $X$. 
For $E \in D(X)$, we define the Mukai vector $v(E)$ of $E$ by $ch(E)\sqrt{td_X}$. 
Then $v(E)=r_E\+ \delta _E \+ s_E$ is in $H^0(X, \bb Z) \+ \mr{NS}(X) \+ H^4(X, \bb Z)$ and we have $r_E=\rank E$, $\delta_E = c_1(E)$ and $s_E = \chi (\mca O_X, E)-r_E$.

Let $\< -,-\>$ be the Mukai paring on $\mca N(X)$:
\[
\< r\+ \delta \+ s, r'\+\delta' \+ s'  \> = \delta \delta '-rs' -r's,
\]
where $r \+ \delta \+ s, r' \+ \delta ' \+ s' \in \mca {N}(X)$. 
Then, by the Riemann-Roch formula, we see  
\[
\chi (E,F) = -\< v(E), v(F)  \>.
\]

We secondly recall the notion of the $\mu$-stability. 
For a torsion free sheaf $F$ and an ample divisor $\omega$, the \textit{slope} $\mu_{\omega}(F)$ is 
defined by $(c_1(F) \cdot \omega)/\rank F$. 
If the inequality $\mu_{\omega}(A) \leq \mu_{\omega}(F)$ holds for any non-trivial subsheaf $A$ of $F$, then $F$ is said to be \textit{$\mu_{\omega}$-semistable}. 
Moreover if the strict inequality $\mu_{\omega}(A) < \mu_{\omega}(F)$ holds for any non-trivial subsheaf $A$ with $\rank A < \rank F$, then 
$F$ is said to be \textit{$\mu_{\omega}$-stable}. 
If $\mr{NS}(X) = \bb Z L$, we write $\mu$-(semi)stable instead of $\mu_L$-(semi)stable. 
The notion of the $\mu_{\omega}$-stability admits the Harder-Narashimhan filtration of $F$ (details in \cite{HL}). 
We define $\mu_{\omega}^+ (F)$ by the maximal slope of semistable factors of $F$, and 
$\mu_{\omega}^- (F)$ by the minimal slope of semistable factors of $F$.

Let $\beta$ be an $\bb R$ divisor and $\omega $ an $\bb R$ ample divisor on $X$\footnote{Originally the notion of twisted stability is defined on projective surfaces. To avoid the complexity we add the assumption that $X$ is a projective K3 surface. }.  
For a pair $(\beta, \omega)$ we define the notion of $(\beta,\omega)$ twisted stability introduce by \cite{MW}.  
For a torsion free sheaf $E$ with $v(E) = r_E \+ \delta _E\+ s_E$, 
we define a polynomial $p_{(\beta, \omega)}(E)$ by 
\[
p_{(\beta, \omega)}(E) := \frac{\omega ^2}{2}\cdot n^2 +  \Bigl(\frac{\delta _E}{r_E}-\beta \Bigr)\omega\cdot n + \frac{s_E}{r_E}- \frac{\delta _E\beta }{r_E}+ \frac{\beta ^2}{2} +1 \in \bb R[n]. 
\]
Suppose that $\omega$ is an integral class and put $\omega = \mca O_X(1)$. Then $p_{(\beta, \omega )}(n)$ is simply given by 
\[
p_{(\beta, \omega)}(E) = -\frac{\< v(\mca O_X(-n)), \exp(-\beta)v(E)\>}{r_E}.  
\]

\begin{dfn}
Let $E$ be a torsion free sheaf on a projective K3 surface $X$. 
$E$ is said to be \textit{$(\beta, \omega)$-twisted (semi)stable} if $p_{(\beta, \omega)}(F) < (\leq )p_{(\beta, \omega)}(E)$ for any nontrivial subsheaf $F$ of $E$. 

Moreover if $\beta =0$ then $E$ is said to be \textit{Gieseker (semi)stable} with respect to $\omega$. 
For a torsion free sheaf $E$, we write $p_{\omega }(E)$ instead of $p_{(0, \omega )}(E)$. 
\end{dfn}

\begin{rmk}
For a torsion free sheaf $E$, we can easily check the following relation between the $\mu_{\omega}$-stability and the $(\beta, \omega)$-twisted stability:
\begin{center}
\resizebox{0.99\hsize}{!}{
$
\mu_{\omega}\mbox{-stable} \Rightarrow (\beta, \omega)\mbox{-twisted stable} \Rightarrow (\beta, \omega)\mbox{-twisted semistable} \Rightarrow \mu_{\omega}\mbox{-semistable}. 
$
}
\end{center}

We also see the following relation between the $\mu_{\omega}$-stability 
and the Gieseker stability:
\begin{center}
\resizebox{0.99\hsize}{!}{
$
\mu_{\omega}\mbox{-stable} \Rightarrow \mbox{Gieseker stable} \Rightarrow \mbox{Gieseker semistable} \Rightarrow \mu_{\omega}\mbox{-semistable}. 
$
}
\end{center}
\end{rmk}

Finally we cite the following lemma which plays an important role when we study the space of stability conditions on abelian or K3 surfaces. 
A prototype of Lemma \ref{hom1} was first proved by Mukai and Bridgeland. 
Finally \cite{HMS} refined it. 

\begin{lem}(\cite[Lemma 2.7]{HMS})\label{hom1}
Let $X$ be an abelian surface or a K3 surface. 
Suppose that $A\to B \to C \to A[1]$ is a distinguished triangle in $D(X)$. 
If $\hom_X^i(A,C)=0$ for any $i\leq 0$ and $\hom_X^j(C,C)=0$ for any $j<0$ then we have the following inequality:
\[
0 \leq  \hom_X^1(A,A) + \hom_X^1(C,C) \leq \hom_X^1(B,B). 
\]
\end{lem}

\section{Review of Bridgeland's work}

In this section we briefly recall the theory of stability conditions. 
The details are in the original articles \cite{Bri} and \cite{Bri2}. 
For a projective K3 surface $X$ we put $\mr{NS}(X)_{\bb R} = \mr{NS}(X) \otimes _{\bb Z}\bb R$ and $\mr{Amp}(X)$ by the set of $\bb R$-ample divisors. 

Let $\mca A$ be the heart of a bounded $t$-structure on the derived category $D(X)$ of $X$ and let $Z$ be a group homomorphism from $K(X)$ to $\bb C$. 
Notice that $K(X)$ is isomorphic to the Grothendieck of the heart $\mca A$. 
The morphism $Z$ is called a \textit{stability function} on $\mca A$ if $Z$ satisfies the following:
\[
0 \neq E \in \mca A \Rightarrow Z(E) =m\exp(\sqrt{-1}\pi \phi_E),
\]
where $m \in \bb R_{>0}$ and $\phi _E$ is in the interval $(0,1]$. 
Then we put $\arg Z(E) = \phi_E$ and call the pair $(\mca A, Z)$ a \textit{stability pair} on $D(X)$. 
If we take a stability pair $(\mca A, Z)$, we can define the notion of $Z$-stability for objects in $\mca A$ as follows: 

\begin{dfn}
Let $(\mca A, Z)$ be a stability pair on $D(X)$ and $E$ in $\mca A$. 
The object $E$ is said to be $Z$-\textit{(semi)stable} if $E$ satisfies $\arg Z(F) <(\leq ) \arg (E)$ for any non-trivial subobject. 
\end{dfn}

By using the notion of $Z$-stability, we define a stability condition on $D(X)$ as follows:

\begin{dfn}
A stability pair $(\mca A, Z) $ is said to be a \textit{stability condition} on $D(X)$ if any $E \in \mca A$ has the filtration 
$0 = E_0 \subset E_1 \subset \cdots E_{n-1}\subset E_n =E$ such that 
$A_i := E_i/E_{i-1}$ ($i=1, 2,\cdots , n$) is $Z$-semistable with $\arg Z(A_1) > \cdots > \arg Z(A_n)$.  
We call such a filtration the \textit{Harder-Narashimhan filtration} of $E$. 
Moreover if $Z$ factors through $\mca N(X)$, $\sigma$ is said to be \textit{numerical}. 
\end{dfn}


Let $\sigma=(\mca A, Z)$ be a stability condition on $D(X)$. 
Then we can define the notion of $\sigma$-stability for any object in $D(X)$\footnote{For a stability pair $(\mca A, Z)$, we can logically define the notion of $\sigma$-stability for objects in $D(X)$. However in the original article \cite{Bri}, the notion of stability of arbitrary objects in $D(X)$ is defined by a stability condition. Thus we follow the original style. }. 
An object $E \in D(X)$ is said to be $\sigma$-\textit{(semi)stable} if there is an integer $n \in \bb Z$ such that $E[n]$ is in $\mca A$ and $E[n]$ is $Z$-(semi)stable. 
We define $\arg Z(E)$ by $\arg Z(E[n])-n$ and call it the \textit{phase} of $E$. 

We put $\mca P(\phi)= \{ E \in \mca D(X) | E \mbox{ is $Z$-semistable with phase } \phi \}\cup \{ 0 \}$. 
Then $\mca P(\phi)$ is an abelian category. 
For an interval $I \subset \bb R$, we define $\mca P(I)$ by the extension closed full subcategory generated by $\mca P(\phi )$ for all $\phi \in I$. 
If for any $\phi \in \bb R$ there is a positive number $\epsilon$ such that $\mca P((\phi-\epsilon, \phi + \epsilon))$ is artinian and noetherian, then the stability condition $\sigma=(\mca A, Z)$ is said to be \textit{locally finite}.

In general we cannot define the argument of $Z(E)$ for $E \in D(X)$. 
However if $E $ is in $\mca A$ (or $\mca A[-1]$) then we can define the argument of $Z(E)$ uniquely 
since the argument $\arg Z(E)$ is in $(0,1]$ (respectively in $(-1,0]$). 

Take a stability condition $\sigma =(\mca A, Z)$ on $D(X)$. 
Then we can easily check that there exists the following sequence of distinguished triangles for an arbitrary object $E \in D(X)$:
\begin{equation*}
\resizebox{0.99\hsize}{!}{
\xymatrix{
0\ar[rr]	&	&	E_1\ar[ld]\ar[rr]	& 	& E_2\ar[r]\ar[ld] & \cdots \ar[r] & E_{n-1}\ar[rr]	&	  &E_n=E ,\ar[ld]\\ 
	&A_1\ar@{-->}[ul]^{[1]}& 		&A_2\ar@{-->}[ul]^{[1]}& 	  &		  &			&A_n\ar@{-->}[ul]^{[1]}& 
}
}
\end{equation*}
where each $A_i$ is $\sigma$-semistable with $\arg Z(A_1) > \cdots > \arg Z(A_n)$.  
One can easily check that the above sequence is unique up to isomorphism. 
We also call this sequence the \textit{Harder-Narashimhan filtration} (for short HN filtration). 
If $E $ is in $\mca A$ then the above filtration is nothing but the filtration defined in Definition 3.2. 
In addition assume that $\sigma$ is locally finite. 
Then for a $\sigma$-semistable object $F$ with phase $\phi$ we have the following sequence of distinguished triangles:
\begin{equation*}
\resizebox{0.99\hsize}{!}{
\xymatrix{
0\ar[rr]	&	&	F_1\ar[ld]\ar[rr]	& 	& F_2\ar[r]\ar[ld] & \cdots \ar[r] & F_{m-1}\ar[rr]	&	  &F_m=F ,\ar[ld]\\ 
	&S_1\ar@{-->}[ul]^{[1]}& 		&S_2\ar@{-->}[ul]^{[1]}& 	  &		  &			&S_m\ar@{-->}[ul]^{[1]}& 
}
}
\end{equation*}
where each $S_j$ is $\sigma$-stable with $\arg Z(S_j) = \phi$. 
We call this filtration a \textit{Jordan-H\"{o}lder filtration} (for short JH filtration). 
We remark that a JH filtration of $F$ is not unique 
but the direct sum $\oplus_{i=1}^{m} S_i$ of all stable factors of $F$ is unique up to isomorphism.

Now we put 
\[
\Stab (X) = \{ \sigma | \sigma \mbox{ is a numerical locally finite stability condition on }D(X) \}.
\]
Bridgeland \cite{Bri2} describes a subset $U(X)$ of $\Stab (X)$. 
We shall recall its definition. 
We put 
\[
\Delta ^+ (X) := \{ v = r \+ \delta \+ s \in \mca N(X) | v^2 =-2, r >0 \}, 
\]
and define a subset $\mca V(X)$ of $\mr{NS}(X)_{\bb R} \times \mr{Amp}(X)$ by 
\begin{multline*}
\mca V(X) := \{  (\beta, \omega) \in\mr{NS}(X)_{\bb R} \times \mr{Amp}(X) | \\
\< \exp(\beta + \sqrt{-1}\omega) , v \> \not\in \bb R_{\leq 0} (\forall v \in \Delta ^+(X) ) \}. 
\end{multline*}
Let $(\beta, \omega) \in \mca V(X)$. 
Then $(\beta ,\omega)$ gives a numerical locally finite stability condition $\sigma_{(\beta, \omega)}=(\mca A_{(\beta, \omega)}, Z_{(\beta ,\omega)})$ in the following way.  
We put $\mca A_{(\beta,\omega)}$ by 
\[
\mca A_{(\beta,\omega )} :=   \{ E^{\bullet} \in D(X)| 
H^i(E^{\bullet})	\begin{cases}
					\in \mca T_{(\beta,\omega)} & (i=0) \\ 
					\in \mca F_{(\beta,\omega)} & (i=-1) \\ 
					= 0 & (i \neq 0,-1) 
					\end{cases}
\},
\]
where 
\begin{eqnarray*}
\mca T_{(\beta,\omega )} &:=& \{ E \in \mr{Coh}(X) | E\mbox{ is a torsion sheaf or }
\mu _{\omega}^-(E/\mr{torsion} ) > \beta \omega  \} \mbox{ and }\\
\mca F_{(\beta,\omega)} &:=& \{ E \in \mr{Coh}(X) | E\mbox{ is torsion free and } 
\mu _{\omega}^+ (E) \leq \beta \omega  \}. 
\end{eqnarray*}
We define a stability function $Z_{(\beta, \omega)}$ by $Z_{(\beta, \omega)}(E) := \< \exp(\beta+ \sqrt{-1}\omega), v(E) \>$. 
Then the pair $\sigma_{(\beta, \omega)}=(\mca A_{(\beta, \omega)}, Z_{(\beta, \omega)})$ gives a numerical locally finite stability condition by \cite{Bri2}.

Then we put $V(X):=\{ \sigma_{(\beta ,\omega)} | (\beta ,\omega )\in \mca V(X) \}$. 
If $\sigma$ is in $V(X)$ then for any closed point $x \in X$, $\mca O_x$ is $\sigma$-stable with phase $1$ by \cite[Lemma 6.3]{Bri2}. 
Let $\tilde{GL}^+(2, \bb R) $ be the universal cover of $GL^+(2, \bb R)$. 
Then $\Stab (X)$ has the right group action of $\tilde{GL}^+(2, \bb R)$ by \cite[Lemma 8.2]{Bri}. 
We put $U(X):= V(X)\cdot \tilde{GL}^+(2, \bb R)$. 
We remark that $U(X)$ is isomorphic to $V(X) \times \tilde{GL}^+(2, \bb R)$.

Let $\sigma$ be in $\Stab (X)$.  
Since $\sigma$ is numerical and the Euler paring is nondegenerate on $\mca N(X)$, we have a natural map
\[
\pi:\Stab (X) \to \mca N(X), \pi(\mca A,Z) \to Z^{\vee},
\]
where $Z(E)= \<  Z^{\vee}, v(E) \>$. 
The map $\pi$ gives a complex structure on $\Stab (X)$. 
In particular each connected component of $\Stab (X)$ is a complex manifold by \cite{Bri}. 
If $\pi(\sigma)$ spans a positive real $2$-plane and is orthogonal to all $(-2)$ vectors in $\mca N(X)$ then $\sigma$ is said to be \textit{good}. 

\begin{prop}(\cite[Proposition 10.3]{Bri2})\label{good}
The special locus $U(X)$ is written by 
\[
U(X) = \{  \sigma \in \Stab(X) | \mca O_x\mbox{ is $\sigma$-stable with a common phase and }\sigma\mbox{ is good}  \}.
\]
\end{prop}

Let us consider the boundary $\partial U(X) := \overline{U(X)}\backslash U(X)$ where $\overline {U(X)}$ is the closure of $U(X)$. 
Then $\partial U(X)$ consists of locally finite union of real codimension $1$ submanifolds by \cite[Proposition 9.2]{Bri2}. 
If $\sigma \in \partial U(X)$ lies on only one these submanifold, then $\sigma$ is said to be \textit{general}. 

\begin{thm}(\cite[Theorem 12.1]{Bri2})\label{wall}
Let $\sigma \in \partial U(X)$ be general. 
Then exactly one of the following holds:

$(A^+):$ There is a spherical locally free sheaf $A$ such that both $A$ and $T_A(\mca O_x)$ are stable factors of $\mca O_x$ for any $x \in X$, where $T_A$ is the spherical twist by $A$.  
Moreover a JH filtration of $\mca O_x$ is given by 
\[
\begin{CD}
A^{\oplus \rank A} @>>> \mca O_x @>>> T_A(\mca O_x) @>>> A^{\oplus \rank A}[1].
\end{CD}
\]
In particular $\mca O_x$ is properly $\sigma$-semistable\footnote{Namely $\mca O_x$ is not $\sigma$-stable but $\sigma$-semistable. } for all $x \in X$ 
and $A$ does not depend on $x \in X$. 

$(A^-):$ There is a spherical locally free sheaf $A$ such that both $A$ and $T_A^{-1}(\mca O_x)$ are stable factors of $\mca O_x$ for any $x \in X$, where $T_A$ is the spherical twist by $A$.  
Moreover a JH filtration of $\mca O_x$ is given by 
\[
\begin{CD}
T_A^{-1}(\mca O_x) @>>> \mca O_x @>>> A^{\oplus \rank A}[2] @>>> T_A^{-1}(\mca O_x)[1].
\end{CD}
\]
In particular $\mca O_x$ is properly $\sigma$-semistable for all $x \in X$ and $A$ does not depend on $x \in X$. 

$(C_k):$ There are a $(-2)$-curve $C$ and an integer $k$ such that $\mca O_x$ is $\sigma$-stable if $x \not\in C$ and $\mca O_x$ is properly $\sigma$-semistable if $x \in C$. 
Moreover a JH filtration of $\mca O_x$ for $x \in C$ is given by 
\[
\begin{CD}
\mca O_C(k+1) @>>> \mca O_x @>>> \mca O_C(k)[1] @>>> \mca O_C(k+1)[1]. 
\end{CD}
\]
\end{thm}

We recall the map $\Phi_*:\Stab(Y) \to \Stab (X)$ induced by an equivalence $\Phi :D(Y) \to D(X)$. 
Let $X$ and $Y$ be projective K3 surfaces, and $\Phi:D(Y) \to D(X)$ an equivalence. 
Then $\Phi$ induces a natural morphism $\Phi_* : \Stab(Y) \to \Stab(X)$ as follows:
\[
\begin{matrix}
	&	\Phi_* : \Stab(Y) \to  \Stab (X),\  	
		\Phi _*\bigl( (\mca A_Y, Z_Y) \bigr)= (\mca A_X, Z_X)	& \\
	& \mbox{where }Z_X(E) = Z_Y \bigl( \Phi ^{-1}(E) \bigr) ,\ 
	\mbox{and }\mca A_X = \Phi ( \mca A_Y ) .
\end{matrix}
\]
Then the following proposition is almost obvious. 

\begin{prop}(\cite[Proposition 6.1]{Kaw})\label{6.1}
Let $X$ and $Y$ be projective K3 surfaces, and $\Phi:D(Y) \to D(X)$ an equivalence. 
For $\sigma \in U(X)$, $\sigma$ is in $\Phi _* (U(Y) )$ if and only if $\Phi (\mca O_y) $ is $\sigma$-stable with the same phase 
for all closed points $y \in Y$. 
\end{prop}

\section{Stability of classically stable sheaves}\label{A}

The goal of this section is to show the $\sigma$-stability of Gieseker stable (or $\mu$-stable) sheaves on a projective K3 surface $X$ for some $\sigma \in \Stab (X)$. 

We first prepare a function (\ref{N}) which plays an important role in this section. 
Let $L_0$ be an ample line bundle on $X$ with $L_0^2=2d$. 
We define a subset $V(X)_{L_0}$ of $V(X) $ by 
\[
V(X)_{L_0} := \{ \sigma _{(\beta, \omega)} \in V(X) | (\beta , \omega) = (x L_0, y L_0) \mbox{ where }(x,y) \in \bb R^2  \}. 
\]
Take an element $\sigma_{(\beta, \omega)} \in V(X)_{L_0}$. 
For an arbitrary object $F \in D(X)$ we put the Mukai vector $v(F)$ by $v(F) = r_F \+ \delta _F \+ s_F$ .
We have the orthogonal decomposition of $\delta _F$ with respect to $L_0$ in $\mr{NS}(X)_{\bb R}$:
\[
\delta _F = n_F L_0 + \nu _F,
\]
where $\nu_F$ is in $\mr{NS}(X)_{\bb R}$ with $\nu _F L_0=0$. 
Then we have 
\begin{eqnarray*}
Z(F)	&=& \frac{v(F)^2}{2 r_F} + \frac{r_F}{2}\Bigl( \omega + \sqrt{-1}\bigl(\frac{\delta _F}{r_F} - \beta \bigr)   \Bigr)^2\\
		 &=& \frac{v(F)^2}{2 r_F} + \frac{r_F}{2}\Bigl( \omega + \sqrt{-1}\bigl(\frac{n_F L_0}{r_F} - \beta \bigr)   \Bigr)^2 - \frac{\nu _F^2}{2 r_F}. 
\end{eqnarray*}
We see that the imaginary part $\mf{Im}Z(F)$ of $Z(F)$ is given by $2\sqrt{-1}yd\lambda _F$ where $\lambda _F = n_F - r_F x$. 
Put $$Z^{L_0} (F) := Z(F) + \frac{\nu _F^2}{2r_F}.$$ 
We define a function $N_{A,E}(x,y)$ for objects $A$ and $E \in D(X)$ and for $\sigma_{(\beta, \omega)} \in V(X)_{L_0}$ by
\begin{equation}
N_{A,E}(x,y) := \lambda _E \mf{Re}Z^{L_0}(A) - \lambda _A \mf{Re}Z^{L_0}(E), 
\label{N}
\end{equation}
where $\mf{Re}$ means taking the real part of a complex number. 

Now suppose that $E$ is a $\mu_{\omega}$-semistable torsion free sheaf 
where $\omega\in \mr{Amp}(X)$. 
For a stability condition $\sigma_{(\beta , \omega)} = (\mca A, Z) \in V(X)$ we see the following:
\begin{itemize}
\item $\mu_{\omega }(E) > \beta \omega \iff E \in \mca A$. 
\item $\mu_{\omega }(E) \leq \beta \omega \iff E \in \mca A[-1]$. 
\end{itemize}

We shall consider following three cases: $\mu_{\omega}(E) > \beta \omega$, $\mu_{\omega}(E) = \beta \omega$ and $\mu_{\omega}(E)  < \beta \omega$. 
We first treat the case when $\mu_{\omega}(E) > \beta \omega $. 

\begin{lem}\label{A.1}
Let $X $ be a projective K3 surface and $\sigma_{(\beta ,\omega)} =(\mca A, Z) \in V(X)$. 
Assume that $A \to E \to F $ is a non trivial distinguished triangle in $\mca A$. 
Namely $A $, $E$ and $F$ are in $\mca A$ (This means that the triangle gives a short exact sequence in $\mca A$. ). 

$(1)$ If $E$ is a torsion free sheaf then $A$ is also torsion free sheaf. 

$(2)$ In addition to $(1)$,  if $E$ is Gieseker stable with respect to $\omega$ then we have $p_{\omega }(A) < p_{\omega }(E)$. 

$(3)$ Let $L_0$ be an ample line bundle. In addition to $(2)$, assume $\sigma_{(\beta, \omega)} \in V(X)_{L_0}$ and $\mu_{\omega }(A)= \mu_{\omega }(E)$. Then we have $\arg Z(A) < \arg Z(E)$. 
\end{lem}

\begin{proof}
Let $H^i(F)$ be the $i$-th cohomology of $F$. Since $F$ is in $\mca A$, $H^i(F)$ is $0$ unless $i$ is $0$ or $-1$. 
Then one can easily check the first assertion by taking cohomologies to the given distinguished triangle and by this fact. 
Hence we start the proof of the second assertion $(2)$. 

We have the following exact sequence of sheaves:
\[
\begin{CD}
0 @>>> H^{-1}(F) @>>> A @>f>> E @>>> H^0(F) @>>> 0.
\end{CD}
\]

Suppose that $H^{-1}(F)$ is not $0$. 
One can easily see 
\[
\mu_{\omega }(H^{-1}(F)) \leq \mu_{\omega }^+(H^{-1}(F)) \leq \beta \omega < \mu_{\omega }^{-}(A) \leq \mu_{\omega }(A). 
\]
Thus we have $\mu_{\omega }(H^{-1}(F)) < \mu_{\omega}(A) < \mu_{\omega }(\Im f)$ where $\Im f$ is the image of the morphism $f :A \to E$. 
Thus we have $p_{\omega }(A) < p_{\omega }(\Im f) \leq p_{\omega }(E)$ since $E$ is Gieseker stable with respect to $\omega$. 

Suppose that $H^{-1}(F) =0$. Then $A$ is a subsheaf of $E$. Since $E$ is Gieseker stable, the assertion is obvious. 

Let us prove the third assertion. 
We put $L_0^2 =2d$. 
For $E$ and $A$ in $D(X)$ we put $v(E) = r_E \+ \delta _E \+ s_E$ and $v(A) = r_A \+ \delta_A \+ s_A$. 
We decompose $\delta _E$ and $\delta _A$ by 
\[
\delta _E = n_E L _0+ \nu _E\mbox{ and }\delta _A = n_A L_0 + \nu _A,
\]
where $\nu _E$ and $\nu _A$ are in $\mr{NS}(X)_{\bb R}$ with $\nu_E L_0 = \nu _A L_0 =0$. 
We remark that both $\nu _E^2 $ and $\nu _A^2$ are semi negative and 
that the number $m_A = \delta _A L_0$ (respectively $m_E = \delta _E L_0$) is an integer. 
Then we have $n_A = m_A/2d$ (respectively $n_E = m_E /2d$) and  
\[
\frac{Z(A)}{r_A} = \frac{v(A)^2}{2r_A^2} + \frac{1}{2}\Bigl( \omega +\sqrt{-1}\bigl( \frac{n_A L_0}{r_A} - \beta \bigr) \Bigr)^2 - \frac{\nu _A^2}{2r_A^2}. 
\]

Now we put $J(A) = \frac{1}{2r_A^2}(v(A)^2-\nu _A^2)$ and $J(E) = \frac{1}{2r_E^2}(v(E)^2-\nu _E^2)$. 
Then we see 
\[
J(A) = \frac{1}{2}\Bigr( \frac{n_A^2L_0^2}{r_A^2}- \frac{s_A}{r_A} \Bigl) \mbox{ and } 
J(E) = \frac{1}{2}\Bigr( \frac{n_E^2L_0^2}{r_E^2}- \frac{s_E}{r_E} \Bigl). 
\]
Since $\mu_{\omega }(E) = \mu_{\omega }(A)$ we see $\frac{Z(E)}{r_E}-J(E) = \frac{Z(A)}{r_A}-J(A)$. 
Thus we see $\arg Z(A) < \arg Z(E)$ if and only if $J(E) < J(A)$. 
Since $p_{\omega}(A) < p _{\omega}(E)$ and $\mu_{\omega}(A) = \mu_{\omega }(E)$, we see 
\begin{equation}
\frac{n_A}{r_A} = \frac{n_E}{r_E} \mbox{ and } \frac{s_A }{r_A} < \frac{s_E}{r_E}. \label{kome}
\end{equation}
Then the inequality $J(E) < J(A)$ follows from the inequality (\ref{kome}). 
Thus we have finished the proof. 
\end{proof}

\begin{lem}\label{A.2}
Let $X$ be a projective K3 surface, let $L_0$ be an ample line bundle on $X$ and let $\sigma_{(\beta ,\omega)}=(\mca A, Z) \in V(X)_{L_0}$. 
Assume that $A \to E \to F$ is a distinguished triangle in $\mca A$ with $\hom_X^0(A,A)=1$ and that 
$E$ is a  torsion free sheaf with $\delta _E = n_E L_0$ for an integer $n_E$ where we put $v(E) = r_E \+ \delta _E \+ s_E$. 

$(1)$ If $v(E)^2=-2$, $\mu_{\omega }(A) < \mu_{\omega }(E)$ and $(\delta _E L_0 - r_E \beta L_0) \leq \frac{\omega ^2}{2}$ then $\arg Z(A) < \arg Z(E)$. 

$(2)$ If $v(E)^2 \geq 0$, $\mu_{\omega }(A) < \mu_{\omega }(E)$ and 
$$(\delta _E L_0 - r_E \beta L_0)\Bigl(\frac{v(E)^2}{2r_E} +1\Bigr)   \leq \frac{\omega ^2}{2}$$
then $\arg Z(A) < \arg Z(E)$. 
\end{lem}

\begin{proof}
We first note that $A$ is a torsion free sheaf by Lemma \ref{A.1}. 
Since there is no $\sigma_{(\beta, \omega)}$-stable torsion free sheaf with phase $1$ (See \cite[Remark 3.5 (1)]{Kaw} or \cite[Lemma 10.1]{Bri2}), 
we see $\mu_{\omega }(A)>\beta \omega$. 
For the Mukai vector $v(A)= r_A \+ \delta _A \+ s_A$ of $A$ we put 
\[
\delta _A = n_A L_0 + \nu _A,
\]
where $\nu_A$ is in $\mr{NS}(X)_{\bb R}$ with $\nu _A L_0=0$. 
Then we have 
\begin{eqnarray*}
Z(A)	&=& \frac{v(A)^2}{2 r_A} + \frac{r_A}{2}\Bigl(\omega +\sqrt{-1}\bigl(\frac{n_A}{r_A}L_0 - r_A \beta  \bigr)   \Bigr) -\frac{\nu _A^2}{2r_A}\\
		&=& Z^{L_0}(A) - \frac{\nu _A^2}{2r_A}. 
\end{eqnarray*}
We note that both $\lambda _A = n_A - r_A x$ and $\lambda _E = n_E - r_E x$ are positive. 
Since $F$ is in $\mca A$, we have $\mf{Im}Z(F) \geq 0$. Thus we see $\lambda _A \leq \lambda _E$. 
Since $\nu _A^2 \leq 0$, we see $\arg Z(A) \leq  \arg Z^{L_0} (A)$. 
Thus it is enough to show that $\arg Z^{L_0}(A) < \arg Z(E)$. 
Since $\mf{Im}Z(A)= \mf{Im}Z^{L_0}(A)= 2yd \lambda _A >0$ and $\mf{Im }Z(E)= 2yd\lambda _E >0$, we see
\[
\arg Z^{L_0}(A) < \arg Z(E) \iff 0< N_{A,E}(x,y). 
\]
Note that $Z^{L_0}(E)=Z(E)$. 

Now we have
\begin{eqnarray*}
N_{A,E}(x,y)	&=&	\lambda _E \mf{Re}Z^{L_0}(A) - \lambda _A \mf{Re}Z(E)\\	
				&=&	\lambda _E \bigl(\frac{v(A)^2}{2r_A}+ dr_Ay^2 - \frac{d \lambda _A^2}{r_A} \bigr) -\lambda _A \bigl(\frac{v(E)^2}{2r_E}+ dr_Ey^2 - \frac{d \lambda _E^2}{r_E} \bigr)\\
				&=&	dy^2(r_A n_E - r_E n_A ) + \lambda _E\frac{v(A)^2}{2 r_A}- \lambda _A \frac{v(E)^2}{2r_E} + d\lambda _A \lambda _E \bigl( \frac{n_E}{r_E}- \frac{n_A}{ r_A} \bigr). 
\end{eqnarray*}
Since $\mu_{\omega}(A) < \mu _{\omega }(E)$ we have $\frac{n_E}{r_E}-\frac{n_A}{r_A}>0$ and $r_A n_E - r_E n_A >0$. 
Since the last term $d\lambda _A \lambda _E \bigl( \frac{n_E}{r_E}- \frac{n_A}{ r_A} \bigr)$ is positive, we have 
\[
N_{A,E}(x,y) > N'_{A,E}(x,y):= dy^2(r_A n_E - r_E n_A ) + \lambda _E\frac{v(A)^2}{2 r_A}- \lambda _A \frac{v(E)^2}{2r_E}.
\]
Since $\hom_X^0(A,A)=1$ we have $v(A)^2\geq -2$. 
Thus we see
\[
N'_{A,E}(x,y) \geq N''_{A,E}(x,y):= dy^2(r_A n_E - r_E n_A ) -\frac{\lambda _E}{ r_A}- \lambda _A \frac{v(E)^2}{2r_E}. 
\]
Hence it is enough to prove $N_{A,E}''(x,y) \geq 0$. 

Let us prove the first assertion $(1)$. 
Since $v(E)^2=-2$ we have 
\begin{eqnarray*}
 N''_{A,E}(x,y)	&=&	dy^2(r_A n_E - r_E n_A ) -\frac{\lambda _E}{ r_A}+\frac{\lambda _A}{r_E} \notag \\ 
 				&>&	dy^2(r_A n_E - r_E n_A ) -\frac{\lambda _E}{ r_A} \label{kimete}
\end{eqnarray*}
We shall show 
\[
0<dy^2(r_A n_E - r_E n_A ) -\frac{\lambda _E}{ r_A} .
\]
Since $n_A = m_A/2d$ with $m_A $ and $d \in \bb Z$,  
we see 
\begin{eqnarray}
\frac{\omega^2}{2} = dy^2 > (\delta _E L_0- r_E \beta L_0) &=& 2d  \lambda _E \notag\\
															&\geq &\frac{2d  \lambda _E}{r_A(2dr_A n_E - r_E m_A)}\notag\\ 
										&=&\frac{\lambda _E}{r_A(r_A n_E - r_E n_A)}. \notag
\end{eqnarray}
Hence we have 
\[
dy^2 (r_A n_E - r_E n_A) - \frac{\lambda _E}{r_A} \geq 0. 
\]
by $r_A n_E -r_E n_A>0$. 
Thus we have proved the assertion.

Let us prove the second assertion. Essentially the proof is the same as the one of the first assertion. 
Assume that $v(E)^2 \geq 0$. 
It is enough to show that $N'' _{A,E}(x,y) \geq 0$. 
Since $0 < \lambda _A \leq \lambda _E$ we have 
\begin{eqnarray}
N''_{A,E}(x,y)	&=&	dy^2(r_A n_E -r_E n_A) - \frac{\lambda _E}{r_A} -\mu_{A}\frac{v(E)^2}{2r_E} \notag \\
				&\geq & dy^2(r_A n_E -r_E n_A) - {\lambda _E} -\mu_{E}\frac{v(E)^2}{2r_E} \label{kimete2}
\end{eqnarray}
Hence it is enough to show that $dy^2(r_A n_E -r_E n_A) - \lambda _E -\mu_{E}\frac{v(E)^2}{2r_E}  \geq 0$. 
Similarly to the first assertion, one can easily prove this inequality by using the assumption 
\[
\frac{\omega ^2}{2}  \geq ( \delta _E L_0 - r_E \beta L_0  )\Bigl( \frac{v(E)^2}{2 r_E} +1\Bigr).
\]
Thus we have proved the second assertion. 
\end{proof}

\begin{cor}\label{A.3}
Notations and assumptions are being as Lemma \ref{A.2}. 
Furthermore we assume that $\mr{NS}(X) = \bb Z L_0$. 

$(1)$ If $v(E)^2=-2$, $\mu_{\omega}(A) < \mu_{\omega}(E)$ and 
$\frac{1}{L_0^2}(\delta _E L_0 - r_E \beta L_0) \leq \frac{\omega ^2}{2}$, 
then $\arg Z(A) < \arg Z(E)$. 

$(2)$ If $v(E)^2 \geq 0$, $\mu_{\omega}(A) < \mu_{\omega}(E)$ and 
\[
\frac{1}{L_0^2}(\delta _E L_0 - r_E \beta L_0) \Bigl( \frac{v(E)^2}{2r_E}+1 \Bigr) \leq \frac{\omega ^2}{2},
\] 
then $\arg Z(A) < \arg Z(E)$. 
\end{cor}

\begin{proof}
We use the same notations as in the proof of Lemma \ref{A.2}. 

Let us prove the first assertion. 
Supposet that $v(E)^2=-2$. 
By using the same argument in the proof of Lemma \ref{A.2}, one can see that 
it is enough to show that 
\begin{equation}
0 \leq dy^2 (r_A n_E - r_E n_A)  -\frac{\lambda _E}{r_A}.\label{darui}
\end{equation}
Since $\mr{NS}(X)= \bb Z L_0$, we see $n_A \in \bb Z$.  
Thus the inequality (\ref{darui}) follows from the assumption $\frac{1}{L_0^2}(\delta _E L_0 - r_E \beta L_0) \leq \frac{\omega ^2}{2}$. 

One can easily prove the second assertion since the proof is essentially as the same as the first assertion. 
In fact one can easily see that it is enough to show 
\begin{equation}
0 \leq dy^2 (r_A n_E - r_E n_A) -\frac{\lambda _E}{r_A}\Bigl( \frac{v(E)^2}{2r_E} +1 \Bigr) ,\label{darui2}
\end{equation}
instead of (\ref{darui}) as above. 
This inequality follows from the assumption $\frac{1}{L_0^2}(\delta _E L_0 - r_E \beta L_0) \Bigl( \frac{v(E)^2}{2r_E}+1 \Bigr) \leq \frac{\omega ^2}{2}.
$
\end{proof}

\begin{thm}\label{A.4}
Let $X$ be a projective K3 surface, $L_0$ an ample line bundle and $\sigma_{(\beta ,\omega)} = (\mca A , Z)\in V(X)_{L_0}$. 
We assume that $E$ is a Gieseker stable torsion free sheaf with respect to $L_0$ with $\mu_{\omega}(E) > \beta \omega $ and that the Mukai vector $v(E)$ is $ r_E \+ \delta _E \+ s_E$ with $\delta _E = n_E L_0$ for some $n_E \in \bb Z$.  

$(1)$ Assume that $v(E)^2=-2$. If $\delta _E L_0 - r_E \beta L_0 \leq \omega^2/2$, then $E$ is $\sigma_{(\beta , \omega)}$-stable.  

$(2)$ Assume that $v(E)^2\geq 0$. If $(\delta _E L_0 - r_E \beta L_0)(\frac{v(E)^2}{2r_E}+1) \leq \omega^2/2$, then $E$ is $\sigma_{(\beta , \omega)}$-stable.  
\end{thm}

\begin{proof}
Suppose to the contrary that $E$ is not $\sigma_{(\beta,\omega)}$-stable. 
Then there is a $\sigma_{(\beta, \omega)}$-stable subobject $A$ of $E$ in $\mca A$ with $\arg Z(A) \geq \arg Z(E)$ and 
we have the following distinguished triangle in $\mca A$:
\[
\begin{CD}
A@>>> E @>>> F @>>> A[1].
\end{CD}
\]
Since $E$ is Gieseker stable with respect to $\omega= yL_0$ we see that $A$ is a torsion free sheaf with $p_{\omega}(A) < p_{\omega}(E)$ by Lemma \ref{A.1}. 
Since $p_{\omega}(A) < p_{\omega }(E)$ we see $\mu_{\omega }(A) \leq \mu_{\omega }(E)$. 
If $\mu_{\omega}(A) = \mu_{\omega}(E)$ then $\arg Z(A) < \arg Z(E)$ by Lemma \ref{A.1}. 
Thus $\mu_{\omega }(A)$ should be strictly smaller than $\mu_{\omega}(E)$. 
Then whether $v(E)^2=-2$ or $v(E)^2 \geq 0$, we see $\arg Z(A) < \arg Z(E)$ by Lemma \ref{A.2}. 
Hence $E$ is a $\sigma_{(\beta,\omega)}$-stable. 
\end{proof}

\begin{cor}\label{A.5}
Notations and assumptions are being as Theorem \ref{A.4}. 
Furthermore we assume that $\mr{NS}(X) = \bb Z L_0$. 

$(1)$ Assume that $v(E)^2=-2$. If $\frac{1}{L_0^2}(\delta _E L_0 - r_E \beta L_0) \leq \omega^2 /2$, then $E$ is $\sigma_{(\beta, \omega)}$-stable. 

$(2)$ Assume that $v(E)^2\geq 0$. If $\frac{1}{L_0^2}(\delta _E L_0 - r_E \beta L_0) (\frac{v(E)^2}{2r_E}+1) \leq \omega^2 /2$, then $E$ is $\sigma_{(\beta, \omega)}$-stable. 
\end{cor}

The proof of Corollary \ref{A.5} is essentially the same as the proof of Theorem \ref{A.4}. 
The difference is to use Corollary \ref{A.3} instead of Lemma \ref{A.2}. 
Hence we omit the proof. 
Next we consider the case ``$\mu_{\omega }(E) = \beta \omega$''.

\begin{prop}\label{A.6}
Let $X$ be a projective K3 surface and $\sigma_{(\beta, \omega)}=(\mca A,Z) \in V(X)$. Assume that the Mukai vector of an object $E \in D(X)$ is $r_E\+ \delta _E \+ s_E$ with $r_E \neq 0$ and $\delta _E \omega /r_E = \beta \omega $. 

$(1)$ If $E$ is a $\mu_{\omega}$-semistable torsion free sheaf then $E$ is $\sigma_{(\beta, \omega)}$-semistable with phase $0$. 

$(2)$ The object $E$ is a $\mu_{\omega}$-stable locally free sheaf if and only if $E$ is $\sigma_{(\beta ,\omega)}$-stable with phase $0$. 
\end{prop}

\begin{proof}
Let us prove the first assertion. 
Since $E$ is $\mu_{\omega}$-semistable, $E$ is in $\mca A[-1]$. 
Since $\mu_{\omega}(E) =\beta \omega$, the imaginary part $\mf{Im}Z(E)$ of $Z(E)$ is $0$. 
Thus the argument of $Z(E)$ is $0$. 

Assume that $E$ is not $\sigma_{(\beta, \omega)}$-semistable. 
Then there is a $\sigma_{(\beta, \omega)}$-semistable object $A \in \mca A[-1]$ such that 
\[
A \subset E \mbox{ in }\mca A[-1] \mbox{ with }\arg Z(A) > \arg Z(E)=0. 
\]
This contradicts the fact that $A$ is in $\mca A[-1]$. 
Hence $E$ is $\sigma_{(\beta, \omega)}$-semistable. 

Let us prove the second assertion. 
We assume that $E$ is a $\mu_{\omega}$-stable locally free sheaf. 
Then $E$ is minimal in $\mca A[-1]$\footnote{Namely $E$ has no non-trivial subobject in $\mca A[-1]$. } by \cite[Theorem 0.2]{Huy}. 
Thus $E$ is $\sigma_{(\beta, \omega)}$-stable with phase $0$.

Conversely we assume that $E$ is $\sigma_{(\beta,\omega)}$-stable with phase $0$. 
Since the rank of $E$ is not $0$, $E$ is a locally free sheaf by \cite[Lemma 10.1 (b)]{Bri2}. 
Since $E$ is in $\mca A[-1]$, we see $E \in \mca F_{(\beta, \omega)}$. 
Thus we have 
\[
\mu_{\omega}(E) \leq \mu_{\omega}^+(E) \leq \beta \omega.
\]
Thus equalities should hold. Hence $E$ is $\mu_{\omega}$-semistable. 

Suppose that $E$ is not $\mu_{\omega}$-stable. 
Then there is a $\mu_{\omega}$-stable subsheaf $A$ of $E$ such that $\mu_{\omega}(A) = \mu_{\omega }(E)$. 
If necessary by taking a saturation, we may assume that the quotient $E/A$ is a torsion free sheaf. 
We remark that $E/A$ is $\mu_{\omega}$-semistable. 
Then $A$ is locally free since $E$ is locally free and $\dim X =2$. 
Since $A$ is a $\mu_{\omega}$-stable locally free sheaf, 
$A$ is $\sigma_{(\beta, \omega)}$-stable with phase $0$. 
Thus the short exact sequence $A \to E \to E/A$ defines a distinguished triangle in $\mca A[-1]$. 
In particular $A$ is a subobject of $E$ in $\mca A[-1]$ with phase $0$. 
This contradicts the fact that $E$ is $\sigma_{(\beta,\omega)}$-stable. 
\end{proof}

Finally we treat the case ``$\mu_{\omega}(E) < \beta \omega$''. 

\begin{lem}\label{A.7}
Let $X$ be a projective K3 surface and $\sigma_{(\beta, \omega)} =(\mca A, Z)\in V(X)$. 
Assume that $F \to E \to A$ is a distinguished triangle in $\mca A[-1]$. 

$(1)$ If $E $ is a torsion free sheaf then $A$ is a torsion free sheaf. 

$(2)$ If $E$ is a $\mu_{\omega}$-stable locally free sheaf then $\mu_{\omega}(E) < \mu_{\omega}(A)$. 
\end{lem}

We remark that the proof of \cite[Lemma 4.4]{Kaw} completely works. 

\begin{proof}
One can easily prove the first assertion by taking cohomologies to the triangle $F \to E \to A$. 
Thus let us prove the second assertion. 
Since $F$, $E$ and $A$ are in $\mca A[-1]$, 
we have an exact sequence of sheaves
\[
\begin{CD}
0@>>> H^0(F) @ >>> E @>f>> A @>>> H^1(F) @>>> 0,
\end{CD}
\]
where $H^i(F)$ is the $i$-th cohomology of $F$. 

Assume that $H^0(F) \neq 0$. 
Since $H^0(F)$ is torsion free, $\rank \Im f< \rank E$, where $\Im f$ is the image of the morphism $f: E \to A$. 
Thus $\mu_{\omega}(E) < \mu_{\omega}(\Im f )$. 
By using the fact $H^1(F) \in \mca T_{(\beta, \omega)}$, one can prove $\mu_{\omega}(\Im f) \leq \mu_{\omega}(A)$. Thus we have $\mu_{\omega}(E)< \mu_{\omega}(A)$.

Assume that $H^0(F)=0$. We write $F$ instead of $H^1(F)$. 
Then $E$ is a subsheaf of $A$. 
If $\rank F$ is not $0$ then we have $\mu_{\omega}(A) \leq \beta \omega < \mu_{\omega}(F)$. 
Thus we have $\mu_{\omega }(E) < \mu_{\omega }(A)$. 
Suppose that $\rank F=0$. If the dimension of the support of $F$ is $1$ then $c_1(F) \omega >0$. 
Hence we see $\mu_{\omega}(E) < \mu_{\omega }(A)$. 
Thus suppose that $F$ is a torsion sheaf with $\dim \mr{Supp}(F)=0$. 
Take a closed point $x \in \mr{Supp}(F)$. 
By taking the right derived functor $\bb R \Hom_X(\mca O_x, -)$ to the triangle $E \to A \to F$, we have the following exact sequence of $\bb C$-vector spaces:
\[
\Hom_X^0(\mca O_x, E) \to \Hom_X^0(\mca O_x, A) \to \Hom_X^0(\mca O _x , F) \to \Hom_X^1(\mca O_x, E) 
\]
Since $E$ is locally free we see $\Hom_X^0(\mca O_x, E) = \Hom_X^1(\mca O_x , E)=0$ by the Serre duality. 
Since $x$ is in the support of $F $, $\Hom_X^0(\mca O_x, F) $ should not be $0$. 
This contradicts the torsion freeness of $A$. 
Hence we have proved the assertion. 
\end{proof}

\begin{lem}\label{A.8}
Let $X$ be a projective K3 surface, let $L_0$ be an ample line bundle,  let $\sigma_{(\beta, \omega )} =(\mca A, Z) \in V(X)_{L_0}$ and let $F \to E \to A$ be a distinguished triangle in $\mca A[-1]$. 
We put $v(E) = r_E \+ \delta _E \+ s_E$. 
Assume that $\hom_X^0(A,A)=1$, both $\rank E$ and $\rank A$ are positive and $\delta _E = n_E L_0$ for some integer $n_E$.

$(1)$ If $v(E)^2=-2$, $\mu_{\omega }(E) < \mu_{\omega }(A) < \beta \omega $ and $r_E \beta L_0 -\delta _E L_0 \leq \omega ^2/2$, 
then $\arg Z(E)< \arg Z(A)$. 

$(2)$ If $v(E)^2 \geq 0$, $\mu_{\omega }(E) < \mu_{\omega }(A) < \beta \omega $ and 
\[
(r_E \beta L_0 -\delta _E L_0) \Bigl(\frac{v(E)^2}{2r_E} +1\Bigr) \leq \frac{\omega ^2}{2},
\]
then $\arg Z(E)< \arg Z(A)$. 
\end{lem}

\begin{proof}
The proof is essentially the same as it of Lemma \ref{A.2}. 
We put $L_0^2 =2d$ and $v(A) = r_A \+ \delta _A \+ s_A$ with $\delta _A = n_A L_0 + \nu _A$, where $\nu_A \in \mr{NS}(X)_{\bb R}$ with $\nu _A L_0=0$. 
If we put $m_A = \delta _A L_0 \in \bb Z$ then we have $n_A = m_A /2 d$. 

Now we have 
\begin{eqnarray*}
Z(A)	&=&	\frac{v(A)^2}{2r_A} + \frac{r_A }{2} \Bigl(\omega + \sqrt{-1}(\frac{n_A L_0 }{r_A} - \beta)  \Bigr)^2  -\frac{\nu _A^2}{2r_A} \\
		&=& Z^{L_0}(A) - \frac{\nu _A^2}{2r_A}
\end{eqnarray*}
Since $\nu _A^2 \leq 0$ we have $\arg Z^{L_0}(A) \leq \arg Z(A)$. 
Thus it is enough to show that $\arg Z(E) < \arg Z^{L_0}(A)$. 
We put $\lambda _E = n_E - r_E x$ and $\lambda _A = n_A - r_A x$. 
We remark that both $\lambda _E$ and $\lambda _A$ are negative and $\lambda_{E} \leq \lambda _A <0$ by the fact $F \in \mca A[-1]$. 
Hence we see
\[
\arg Z(E) < \arg Z^{L_0}(A) \iff N_{A,E}(x,y) < 0.
\]

Now we have 
\[
N_{A,E}(x,y) = dy^2 (r_A n_E - r_E n_A ) + d \lambda _A \lambda _E (\frac{n_E }{r_E}-\frac{n_A }{r_A}) + \frac{v(A)^2}{2r_A}\lambda _E - \frac{v(E)^2}{2r_E}\lambda _A.
\]
Since $\mu_{\omega }(E) < \mu_{\omega }(A)$ we see $r_A n_E - r_E n_A <0$. Thus we have 
\[
N_{A,E}(x,y) < N_{A,E}'(x,y) := dy^2 (r_A n_E - r_E n_A ) +  \frac{v(A)^2}{2r_A}\lambda _E - \frac{v(E)^2}{2r_E}\lambda _A. 
\]
Since $\hom_X^0(A,A)=1$ we have $v(A)^2\geq -2$. 
Thus we see 
\[
N_{A,E}'(x,y) \leq  N_{A,E}''(x,y) := dy^2 (r_A n_E - r_E n_A ) - \frac{\lambda _E}{r_A} - \frac{v(E)^2}{2r_E}\lambda _A.
\]
Hence it is enough to show $N_{A,E}''(x,y)\leq  0$. 

Assume that $v(E)^2=-2$, then  
\begin{eqnarray*}
N_{A,E}''(x,y)	&=&	dy^2 (r_A n_E - r_E n_A ) - \frac{\lambda _E}{r_A} + \frac{\lambda _A}{r_E} \\
				&\leq &dy^2 (r_A n_E - r_E n_A ) - \frac{\lambda _E}{r_A} 
\end{eqnarray*}
Hence it is enough to show that 
\begin{equation}
dy^2 (r_A n_E - r_E n_A ) - \frac{\lambda _E}{r_A} \leq 0.  \label{temp}
\end{equation}

Recall that $n_A = m_A/2d$ for some integer $m_A$ and $d \in \bb Z$. 
Then the inequality (\ref{temp}) follows from the assumption  
\begin{equation}
r_E \beta L_0 -\delta _E L_0 \leq \frac{\omega ^2}{2},\notag
\end{equation}
Thus we have finished the proof. 

Assume that $v(E)^2 \geq 0$. 
Then we have 
\begin{eqnarray*}
N_{A,E}''(x,y)	&=&	dy^2 (r_A n_E - r_E n_A ) - \frac{\lambda _E}{r_A} - \frac{v(E)^2}{2r_E}\lambda _A \\
				&\leq & dy^2 (r_A n_E - r_E n_A ) - \lambda _E - \frac{v(E)^2}{2r_E}\lambda _E .
\end{eqnarray*}
Hence it is enough to show that 
\begin{equation}
dy^2 (r_A n_E - r_E n_A ) - \lambda _E - \frac{v(E)^2}{2r_E}\lambda _E \leq 0. \label{kome2}
\end{equation}
The inequality (\ref{kome2}) is equivalent to the following inequality 
\begin{equation}
\frac{-\lambda _E}{r_A (r_E n_A - r_A n_E)}\Bigl( \frac{v(E)^2}{2r_E} +1  \Bigr ) \leq  dy^2. \label{kimete3} 
\end{equation}
The last inequality (\ref{kimete3}) follows from the assumption
\[
(r_E \beta L_0 -\delta _E L_0) \Bigl(\frac{v(E)^2}{2r_E} +1\Bigr) \leq \frac{\omega ^2}{2}. 
\]
Thus we have proved the assertion. 
\end{proof}

Similarly to the case of Corollary \ref{A.5}, we have the following corollary. 
We omit the proof since the proof is as the same as the proof of Lemma \ref{A.8}. 

\begin{cor}\label{A.9}
Notations and assumptions are being as Lemma \ref{A.8}. 
Furthermore we assume that $\mr{NS}(X) = \bb Z L_0$. 

$(1)$ Assume that $v(E)^2 =-2$, $\mu_{\omega }(E) < \mu_{\omega }(A) < \beta \omega$ and $\frac{1}{L_0^2}(r_E \beta L_0  - \delta _E L_0) \leq \omega ^2 /2$.  Then we have $\arg Z(E) < \arg Z(A)$. 

$(2)$ Assume that $v(E)^2 \geq 0$, $\mu_{\omega }(E) < \mu_{\omega }(A) < \beta \omega$ and 
\[
\frac{1}{L_0^2}(r_E \beta L_0  - \delta _E L_0) \Bigl(\frac{v(E)^2 }{2 r_E}+1  \Bigr) \leq \omega ^2 /2.
\]  
Then we have $\arg Z(E) < \arg Z(A)$. 
\end{cor}

\begin{thm}\label{A.10}
Let $X $ be a projective K3 surface, $L_0$ an ample line bundle and $\sigma_{(\beta, \omega )}=(\mca A, Z) \in V(X)_{L_0}$. 
Assume that $E$ is a $\mu_{L_0}$-stable locally free sheaf. 
We put $v(E) = r_E \+ \delta _E \+ s_E$. 
Assume that $\delta_E = n_E L_0 $ and $\mu_{\omega }(E) < \beta \omega $ where $n_E \in \bb Z$.  

$(1)$ Assume that $v(E)^2=-2$. If $(r_E \beta L_0 - \delta _E L_0 ) < \omega ^2 /2$ then $E$ is $\sigma _{(\beta, \omega)}$-stable. 

$(2)$ Assume that $v(E)^2 \geq 0$. If $(r_E \beta L_0 - \delta _E L_0 )\Bigl(\frac{v(E)^2}{2r_E}+1 \Bigr) < \omega ^2 /2$ then $E$ is $\sigma _{(\beta, \omega)}$-stable. 
\end{thm}

\begin{proof}
Since $\mu_{\omega }(E) < \beta \omega$, $E$ is in $\mca A[-1]$ and $\arg Z(E) < 0$. 
Suppose to the contrary that $E$ is not $\sigma_{(\beta, \omega)}$-stable. 
Then there is a $\sigma_{(\beta, \omega)}$-stable object $A$ such that $A$ is a quotient of $E$ in $\mca A[-1]$ with $\arg Z(A ) \leq \arg Z(E) $. 
Thus we have a distinguished triangle in $\mca A[-1]$:
\[
\begin{CD}
F @>>> E @>>> A @>>> F[1].
\end{CD}
\]
By Lemma \ref{A.7}, $A$ is a torsion free sheaf with $\mu_{\omega }(E) < \mu_{\omega }(A)$. 
Since $A$ is in $\mca A[-1]$, we see $\mu_{\omega }(A)  \leq \beta \omega$. 
If $\mu_{\omega }(A) = \beta \omega $, then the imaginary part of $Z(A)$ is $0$. 
Thus $A$ is $\sigma_{(\beta,\omega )}$-stable with phase $0$. 
This contradicts $\arg Z(A) < \arg Z(E) < 0$. 
Hence $\mu_{\omega }(A) < \beta \omega $. 
Then we see $\arg Z(E) < \arg Z(A) $ by Lemma \ref{A.8} whether $v(E)^2 =-2$ or $v(E) ^2 \geq 0$. 
This is a contradiction. 
Thus $E$ is $\sigma_{(\beta, \omega )}$-stable. 
\end{proof}

\begin{cor}\label{A.11}
Notations and assumptions are being as Theorem \ref{A.10}. 
Furthermore we assume that $\mr{NS}(X) = \bb Z L_0$. 

$(1)$ Assume that $v(E)^2=-2$. If $\frac{1}{L_0^2}(r_E \beta L_0 - \delta _E L_0 ) < \omega ^2 /2$ then $E$ is $\sigma _{(\beta, \omega)}$-stable. 

$(2)$ Assume that $v(E)^2 \geq 0$. If $\frac{1}{L_0^2}(r_E \beta L_0 - \delta _E L_0 )\Bigl(\frac{v(E)^2}{2r_E}+1 \Bigr) < \omega ^2 /2$ then $E$ is $\sigma _{(\beta, \omega)}$-stable. 
\end{cor}

The proof is essentially as the same as it of Theorem \ref{A.10}. 
One can easily Corollary \ref{A.11} by using Corollary \ref{A.9} instead of \ref{A.8}. 
Hence we omit the proof.

\section{First application}\label{B}
The goal of this section is to prove Theorem \ref{B.4} as an application of Corollaries \ref{A.5} and \ref{A.11}. 
We shall give a classification of fine moduli spaces of Gieseker stable torsion free sheaves on a projective K3 surface with Picard number one. 
In this section the pair $(X, L) $ is called a \textit{generic K3} if $X$ is a projective K3 surface and $\mr{NS}(X)$ is generated by an ample line bundle.

We shall start this section with an easy observation. 
Suppose that $E$ is a Gieseker stable torsion free sheaf on a generic K3 $(X, L)$. 
Since $E$ is Gieseker stable we have $v(E) ^2 \geq -2$. 
Assume that $v(E) ^2 =-2$. Then $\hom_X^1 (E,E) =0$. 
Thus $E$ is a spherical sheaf. 
It is known that $E$ is $\mu$-stable locally free sheaf (For instance see \cite[Proposition 5.2]{Kaw}). 
Thus the notion of $\mu$-stability is equivalent to the notion of Gieseker stability if $v(E)^2=-2$. 

Next we consider the case $v(E) ^2 \geq 0$. 
We write down the following proposition which plays a key roll in this section. 

\begin{prop}\label{B.1}

Let $X$ be a projective K3 surface and $L$ an ample line bundle. 
Assume that $E$ is a Gieseker stable torsion free sheaf with respect to $L$ with $v(E)^2 =0$. 

$(1)$ Assume that $\rank E > 1$. If $E$ is $\mu$-stable with respect to $L$ then $E $ is locally free. 

$(2)$ Assume that $\mr{NS}(X) = \bb Z L$. 
If $E$ is locally free then $E$ is $\mu$-stable with respect to $L$. 

In particular if $\mr{NS}(X)= \bb Z L$ and $\rank E > 1$ then the following holds:
If $E$ is not $\mu$-stable locally free then $E$ is neither $\mu$-stable nor locally free. 
\end{prop}

\begin{proof}
The first assertion was proved in the step vii) in the proof of \cite[Proposition 4.1]{Huy}. 
%

Hence, let us prove the second assertion. 
For any $F \in D(X)$ we put $v(F) = r_F \+ \delta _F \+ s_F $. 
Assume that $E$ is not $\mu$-stable. 
Then there is a $\mu$-stable subsheaf $A$ of $E$ such that $\mu_{L}(A) = \mu_{L}(E)$ and the quotient $E/A$ is torsion free. 
Since $E$ is locally free, $A$ is also locally free. 
We remark that $p_L (A) < p_L (E)$ since $E$ is Gieseker stable. 
We remark that $\frac{s_A}{r_A} < \frac{s_E }{r_E}$ by $\mu_L (A) = \mu_L (E)$. 
Hence we have
\[
0=\frac{v(E)^2 }{r_E^2} = \frac{\delta _E^2}{r_E ^2 } - 2 \frac{s_E }{r_E } < \frac{\delta _A^2}{r_A ^2 } - 2 \frac{s_A }{r_A } = \frac{v(A)^2 }{r_A^2}.  
\]
Thus $v(A) ^2 >0$. 

We choose $\sigma _{(\beta, \omega )}\in V(X)$ such that $\mu_{\omega }(E) = \mu_{\omega }(A) = \beta \omega$. 
Then $E$ is $\sigma _{(\beta, \omega)}$-semistable with phase $0$ and $A$ is $\sigma_{(\beta , \omega )}$-stable with phase $0$ by Proposition \ref{A.6}. 
Since $\sigma_{(\beta, \omega)}$ is locally finite we have a distinguished triangle 
\[
\begin{CD}
A' @>>> E @>>> E/A' ,
\end{CD}
\]
where all stable factors of $A'$ are $A$ and $\hom_X^0(A', E/A')=0$. 
Then by Lemma \ref{hom1}, we see $\hom_X^1(A' ,A') \leq 2$. 
Thus $v(A' )^2\leq 0$. 
However, since $A'$ is an extension of $A$, we have $v(A') = \ell v(A)$ for some $\ell \in \bb N$. Thus $v(A')^2 = \ell ^2 v(A)^2 >0$. 
This is contradiction. 
Hence $E$ is $\mu$-stable. 
\end{proof}

Suppose that $(X,L)$ is a generic K3 and take an element $v = r\+ \delta \+ s \in \mca N(X)$ with $r > 0$ and $v^2 \geq -2$. 
We define subsets of $V(X)$ depending on $v$. 

Case 1: $v^2 =-2$. 
\begin{eqnarray*}
V_{v}^+ &:=& \{ \sigma _{(\beta, \omega)}\in V(X) | \beta \omega < \frac{\delta \omega}{r}, \frac{1}{L^2}(\delta  L - r \beta L) \leq \frac{\omega ^2}{2}  \}. \\
V_{v}^0 &:=& \{ \sigma _{(\beta, \omega)}\in V(X) | \beta \omega = \frac{\delta \omega}{r} \}. \\
V_{v}^- &:=& \{ \sigma _{(\beta, \omega)}\in V(X) | \beta \omega > \frac{\delta \omega}{r}, \frac{-1}{L^2}(\delta  L - r \beta L) \leq \frac{\omega ^2}{2}  \}.
\end{eqnarray*}

Case 2: $v^2 \geq 0$. 
\begin{eqnarray*}
V_{v}^+ &:=& \{ \sigma _{(\beta, \omega)}\in V(X) | \beta \omega < \frac{\delta \omega}{r}, \frac{1}{L^2}(\delta  L - r \beta L)\Bigl( \frac{v^2}{2r}+1 \Bigr) \leq \frac{\omega ^2}{2}  \}. \\
V_{v}^0 &:=& \{ \sigma _{(\beta, \omega)}\in V(X) | \beta \omega = \frac{\delta \omega}{r}\}. \\
V_{v}^- &:=& \{ \sigma _{(\beta, \omega)}\in V(X) | \beta \omega > \frac{\delta \omega}{r}, \frac{-1}{L^2}(\delta  L - r \beta L)\Bigl( \frac{v^2}{2r}+1 \Bigr) \leq \frac{\omega ^2}{2}  \}.
\end{eqnarray*}

For instance, take a Gieseker stable torsion free sheaf $E$ on $(X, L)$ with $v(E)^2=0$ and put $v= v(E)= r\+ \delta \+ s $. 
Then the picture of the sets $V^+_v$, $V^0_v$ and $V^-_v$ are given by the following. 
\begin{center}
\includegraphics[width=90mm,height=60mm,clip]{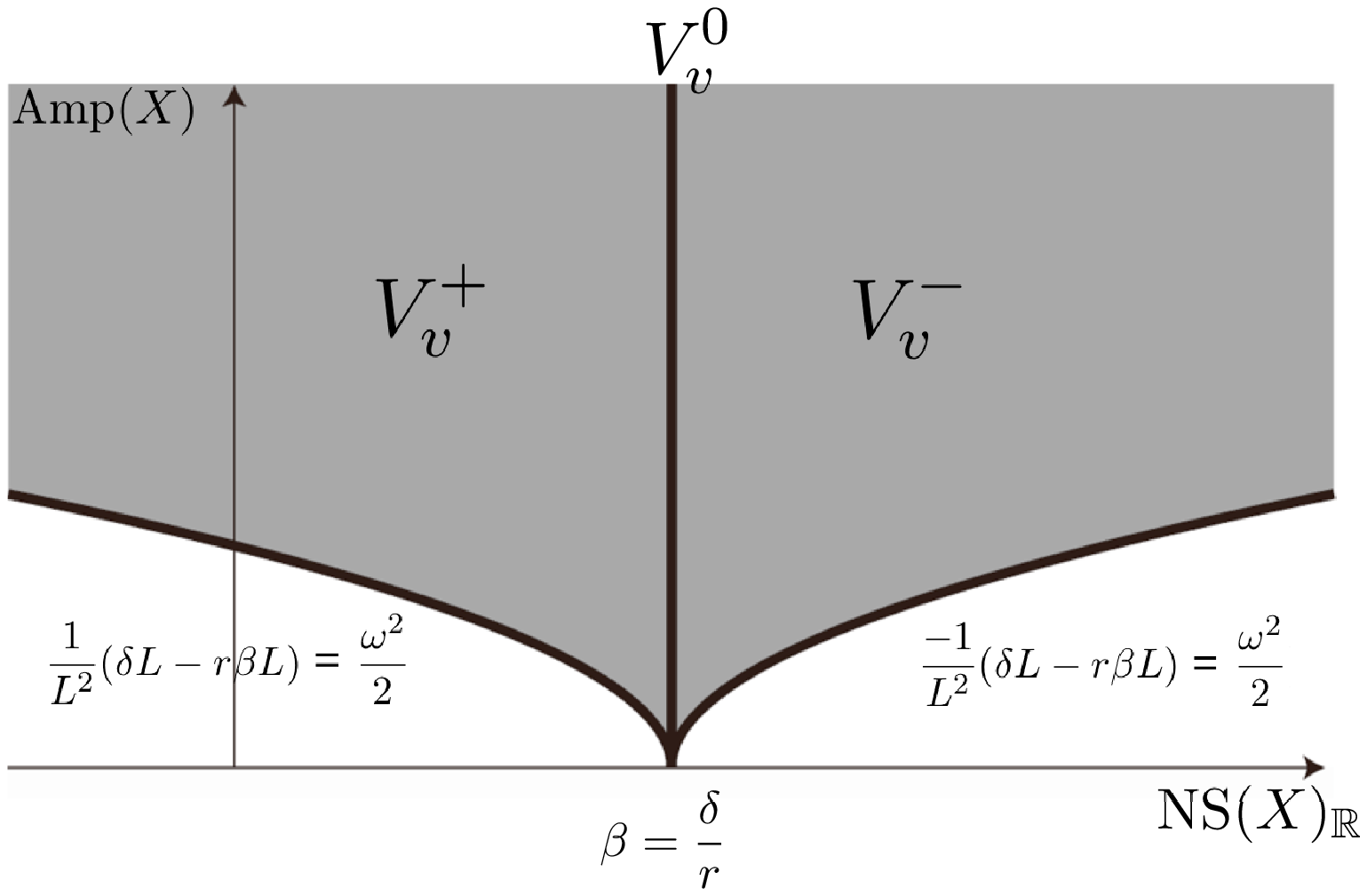}
\end{center}
In Proposition \ref{B.2} (below), 
we show that the set $V_v^0$ is a ``wall'' if and only if $E$ is not a $\mu$-stable locally free but Gieseker stable torsion free.

\begin{prop}\label{B.2}
Let $(X,L) $ be a generic K3 and $E$ a Gieseker stable torsion free sheaf with $v(E)^2 \geq 0$. 

$(1)$ If the sheaf $E$ is not locally free then $E$ is not $\sigma$-semistable for any $\sigma \in V^-_{v(E)}$. 

$(2)$ If the sheaf $E$ is not $\mu$-stable then $E$ is not $\sigma$-semistable for any $\sigma \in V^-_{v(E)}$. 

$(3)$ Take an arbitrary $\sigma \in V^-_{v(E)}$. For the sheaf $E$, the following three conditions are equivalent$:$ $(a)$ $E$ is $\sigma$-stable, $(b)$ $E$ is $\sigma$-semistable and $(c)$ $E$ is $\mu$-stable and locally free. 
\end{prop}

\begin{proof}
For an object $F \in D(X)$ we put $v(F) = r_F \+ \delta _F \+ s_F$. 
Take an arbitrary element $\sigma_0=(\mca A, Z) \in V^-_{v(E)}$.

Let us prove the first assertion $(1)$. 
Suppose to the contrary that $E$ is $\sigma_0$-semistable. 
Since $E$ is not locally free, we have the following distinguished triangle by taking double dual of $E$:
\[
\begin{CD}
S[-1] @>>> E @>>> E^{\vee \vee },
\end{CD}
\]
where $S = E^{\vee \vee }/E$. 
Note that $S$ is a torsion sheaf with $\dim \mr{Supp}(S)=0$. 
Hence $S[-1]$ is $\sigma_0$-semistable with phase $0$. 
Since $\sigma_0 \in V_{v(E)}^-$ we see $\mf{Im}Z(E) < 0$. 
Hence $E$ is $\sigma$-semistable with phase $\phi \in (-1,0)$. 
Thus $\arg Z(E) < \arg Z(S[-1])$ and $\Hom _X(S[-1], E) $ should be $0$. 
This contradicts the above triangle. 
Hence $E$ is not $\sigma_0$-semistable.

Let us prove the second assertion $(2)$. 
Suppose to the contrary that $E$ is $\sigma_0$-semistable. 
Since $E$ is not $\mu$-stable, there is a torsion free quotient $A$ of $E$ such that $A$ is $\mu$-stable with $\mu_{L}(A)=\mu_L(E)$. 
Since $E$ is Gieseker stable we have $p_L (E) < p_L(A)$. 
Thus we see $\frac{s_E }{r_E} < \frac{s_A }{r_A}$. 
Moreover we can assume that $A$ is locally free. 
In fact if necessary it is enough to take the double dual of $A$. 
Then we see that $\mu_{L}(A^{\vee \vee}) = \mu_L(E)$, $\frac{s_E}{r_E }< \frac{s_{A^{\vee \vee}}}{r_A^{\vee \vee}}$ and $A^{\vee \vee}$ is $\mu$-stable. 
Thus we can assume that $A$ is a $\mu$-stable locally free sheaf.  
Note that $\Hom_X^0(E,A)\neq 0$.

We show $V_{v(E)}^{-} \subset V_{v(A)}^{-}$. 
Note that 
\begin{eqnarray}
r_A \beta L - \delta _A L	&=& r_A (\beta L - \frac{\delta_A }{r_A}L) \notag\\
						&=& r_A (\beta L - \frac{\delta_E }{r_E}L) \notag \\
						&<& r_E (\beta L - \frac{\delta_E }{r_E}L)=r_E \beta L - \delta _E L. \label{kimo}
\end{eqnarray}
Here we use the fact $\mr{NS}(X) = \bb Z L$ in the second inequality. 

Since $A$ is $\mu$-stable we have $v(A)^2\geq -2$. 
By the definition of $V_{v(A)}^-$, we have to consider two cases. 
We first assume that $v(A)^2 =-2$. 
Since $v(E)^2 \geq 0$, we have 
\[
1 \leq \frac{v(E)^2}{2r_E} +1.
\]
Then we see 
\[
r_A \beta L - \delta _A L < (r_E \beta L - \delta _E L) \Bigl( \frac{v(E)^2}{2r_E}+1  \Bigr).
\]
Hence  we see $V_{v(E)}^- \subset V_{v(A)}^-$ by the definition of $V_{v(E)}^-$. 

Next suppose that $v(A)^2 \geq 0$. 
Then by using the fact that $\mr{NS}(X) = \bb Z L$ we have 
\begin{eqnarray}
\frac{v(A)^2}{r_A}	&=&	\Bigl( \frac{\delta _A^2}{r_A^2} - 2\frac{s_A}{r_A}  \Bigr) r_A \notag \\
					&<&	\Bigl( \frac{\delta _E^2}{r_E^2} - 2\frac{s_E}{r_E}  \Bigr) r_A \notag \\
					&<&	\Bigl( \frac{\delta _E^2}{r_E^2} - 2\frac{s_E}{r_E}  \Bigr) r_E = \frac{v(E)^2}{r_E}. \label{kimo2}
\end{eqnarray}
By two inequalities (\ref{kimo}) and (\ref{kimo2}) we see 
\[
(r_A \beta L - \delta _A L) \Bigl( \frac{v(A)^2 }{2 r_A } +1  \Bigr) < (r_E \beta L - \delta _E L) \Bigl( \frac{v(E)^2 }{2 r_E } +1  \Bigr)
\]
Thus we have proved $V_{v(E)}^- \subset V_{v(A)}^-$. 

Recall that $A$ is a $\mu$-stable locally free sheaf. 
Since the stability condition $\sigma _0$ is in $ V_{v(A)}^-$, $A$ is $\sigma_0$-stable by Corollary \ref{A.11}. 
Now we have
\begin{eqnarray*}
\frac{Z(A)}{r_A}	&=&	\frac{v(A)^2}{2r_A^2} + \frac{1}{2}\Bigl(\omega  + \sqrt{-1}\bigl( \frac{\delta _A}{r_A}-\beta \bigr)\Bigr)^2\\
					&=&	\frac{v(A)^2}{2r_A^2} -\frac{v(E)^2}{2r_E^2}+ \frac{v(E)^2}{2r_E^2} + \frac{1}{2}\Bigl(\omega  + \sqrt{-1}\bigl( \frac{\delta _E}{r_E}-\beta \bigr)\Bigr)^2\\
					&=& \frac{Z(E)}{r_E} +\frac{v(A)^2}{2r_A^2} -\frac{v(E)^2}{2r_E^2}.
\end{eqnarray*}
Here we used the fact $\mr{NS}(X)= \bb Z L $ in the second equality. 
Since $\mu_{L}(A) = \mu_{L}(E)$, $\frac{s_E }{r_E} < \frac{s_A}{r_A}$ and $\mr{NS}(X) = \bb Z L$, we see that $\frac{v(A)^2}{2r_A^2} -\frac{v(E)^2}{2r_E^2}$ is a negative number. 
Hence we see 
\[
\arg \frac{Z(A)}{r_A} < \arg \frac{Z(E)}{r_E}.  
\]
This contradicts $\Hom_{X}^0(E,A)\neq0$ since both $A$ and $E$ are $\sigma_0$-semistable. 
Thus $E$ is not $\sigma_0$-semistable.

Let us prove the third assertion. 
We claim $(a) \Rightarrow (b) \Rightarrow (c) \Rightarrow (a)$. 
The first claim $(a) \Rightarrow (b)$ is trivial. 
The second claim $(b) \Rightarrow (c)$ follows from the contrapositions of Proposition \ref{B.2} $(1)$ and $(2)$. 
The third claim $(c) \Rightarrow (a)$ is nothing but Corollary \ref{A.11}. 
Thus we have finished the proof. 
\end{proof}

Take a stability condition $\sigma_{(\beta, \omega)} \in V(X)$ and  a $\mu$-semistable torsion free sheaf $E$ with $\mu_{\omega}(E) = \beta \omega$. 
By Proposition \ref{A.6}, if $E$ is not a $\mu$-stable locally free sheaf,  
then $E$ is properly $\sigma$-semistable. 
Hence it makes sense to consider a Jordan-H\"{o}lder filtration of $E$ with respect to $\sigma_{(\beta, \omega)}$. 

\begin{lem}\label{B.3}
Let $X$ be a projective K3 surface. Take a $\sigma_{(\beta, \omega)} \in V(X)$. 
Assume that $E$ is a $\mu_{\omega}$-semistable torsion free sheaf with $\mu_{\omega}(E)=\beta \omega$ and the filtration
\[
0=E_0 \subset E_1 \subset E_2 \subset \cdots \subset E_{n-1}\subset E_n =E
\] 
is a Jordan-H\"{o}lder filtration of $E$ with respect to $\mu_{\omega}$-stability. 
Namely $A_i = E_i/ E_{i-1}$ $(i=1, 2, \cdots ,n)$ is a $\mu_{\omega}$-stable torsion free sheaf with $\mu_{\omega }(A_i)=\mu_{\omega}(E)$. 
Then $\sigma$-stable factors of $E$ consist of all $A_i^{\vee \vee}$ and $\sigma$-stable factors of $A_i^{\vee \vee}/ A_i[-1]$ $(i=1,2,\cdots ,n)$. 
\end{lem}

\begin{proof}
We put $\sigma = \sigma _{(\beta,\omega)}$. 
All $A_i$ ($i = 1,2, \cdots k$) are $\sigma$-semistable by Proposition \ref{A.6}. 
If we obtain JH filtrations of $A_i$, we can construct a JH filtration of $E$ by combining JH filtrations of $A_i$. 
Hence it is enough to prove the assertion for $\mu$-stable torsion free sheaves. 

Suppose that $A$ is a $\mu_{\omega}$-stable torsion free sheaf with $\mu_{\omega }(A) = \beta \omega $ and put $S_A = A^{\vee \vee} /A$. 
Then we have a distinguished triangle:
\[
\begin{CD}
S_A[-1] @>>> A @>>> A^{\vee \vee} @>>> S_A. 
\end{CD}
\]
Since the dimension of the support of $S_A$ is $0$, there are finite closed points $\{ x_1, x_2 , \cdots ,x_k \}$\footnote{There may be $i$ and $j$ in $\{1,2, \cdots, k\}$ so that $x_i=x_j$. } such that 
\[
\resizebox{0.99\hsize}{!}{
\xymatrix{
0\ar[rr]	&	&	\mca O_{x_1}\ar[ld]\ar[rr]	& 	& F_2\ar[r]\ar[ld] & \cdots \ar[r] & F_{k-1}\ar[rr]	&	  &F_k=S_A .\ar[ld]\\ 
	&\mca O_{x_1}\ar@{-->}[ul]^{[1]}& 		&\mca O_{x_2}\ar@{-->}[ul]^{[1]}& 	  &		  &			&\mca O_{x_k}\ar@{-->}[ul]^{[1]}& 
}
}
\]
Since $\mca O_{x_i}$ ($i=1,2, \cdots k$) and $A^{\vee \vee}$ are $\sigma$-stable, these are $\sigma$-stable factors of $A$ and the JH filtration of $A$ with respect to $\sigma$ is given by  
\[
\resizebox{0.99\hsize}{!}{
\xymatrix{
0 \ar[rr]	&	&	\mca O_{x_1}[-1]\ar[ld]\ar[rr]	& 	& F_2[-1]\ar[r]\ar[ld] & \cdots \ar[r] & F_{k-1}[-1]\ar[rr]	&	  &S_A[-1] \ar[ld]\ar[rr] & & A.\ar[ld]\\ 
	&\mca O_{x_1}[-1]\ar@{-->}[ul]^{[1]}& 		&\mca O_{x_2}[-1]\ar@{-->}[ul]^{[1]}& 	  &		  &			&\mca O_{x_k}[-1]\ar@{-->}[ul]^{[1]}& &A^{\vee \vee }\ar@{-->}[lu]^{[1]}&
}
}
\]
Thus we have finished the proof. 
\end{proof}

In the next theorem, we give a classification of moduli spaces of Gieseker stable torsion free sheaves on a generic K3 $(X,L)$. 
Let $Y$ be the fine moduli space of Gieseker stable torsion free sheaves with Mukai vector $v=r \+ \delta \+ s$ and let $\mca E$ be a universal family of the moduli $Y$. 
We define an equivalence $\Phi_{\mca E}:D(Y) \to D(X)$ by 
\[
\Phi_{\mca E}(-)= \bb R \pi _{X*} (\mca E \stackrel{\bb L}{\otimes}\pi_Y^*(-) ), 
\]
where $\pi _X$ (respectively $\pi_Y$) is the projection $X \times Y \to X$ (respectively $X \times Y \to Y$). 
To avoid the complexity in notations, we write $V^+$ (respectively $V^0$ and $V^-$) instead of $V_{v(\Phi(\mca O_y))}^+$ (respectively $V_{v(\Phi(\mca O_y))}^0$ and $V_{v(\Phi(\mca O_y))}^-$) for the given equivalence $\Phi_{\mca E}:D(Y) \to D(X)$.

\begin{thm}\label{B.4}
Notations are being as above. 

$(1)$ If $r$ is not a square number then $Y$ is the fine moduli space of $\mu$-stable locally free sheaf. 

$(2)$ Assume that $r$ is a square number. Then one of the following two cases occurs$:$
\begin{itemize}
\item[$(a)$] $Y$ is the fine moduli space of $\mu$-stable locally free sheaves. 
\item[$(b)$] $Y$ is the fine moduli space of properly Gieseker stable torsion free sheaves and $Y$ is isomorphic to $X$. 
Moreover $\Phi_{\mca E} $ is the spherical twist by a spherical locally free sheaf up to an isomorphism $Y \to X$. 
\end{itemize}
\end{thm}

\begin{proof}
We note that $Y$ is the fine moduli space of properly Gieseker stable torsion free sheaves or the moduli of $\mu$-stable locally free sheaf by Proposition \ref{B.1}. 
Let $\Phi_{\mca E*}$ be a natural map $\Phi_{\mca E *}:\Stab (Y) \to \Stab (X)$ induced by $\Phi_{\mca E}$. 
We put $\mca E_y = \Phi_{\mca E}(\mca O_y)$. 
Then for any $\sigma \in V^+$, $\mca E_y$ is $\sigma$-stable by Corollary \ref{A.5}, and the phase of $\mca E_y$ does not depend on $y \in Y$. 
Hence we see $V^+ \subset \Phi_{\mca E*}U(Y)$. 
By Proposition \ref{A.6} it is enough to show that $V^0 \cap \Phi_{\mca E *}U(Y)\neq \emptyset$. 

Suppose to the contrary that $V^0 \cap \Phi_{\mca E *}U(Y) = \emptyset$. 
We first show that $V^0$ is contained in the boundary $\partial \Phi_{\mca E *}U(Y) $ under the assumption $V^0 \cap \Phi_{\mca E *}U(Y) = \emptyset$. 
Since $V^0$ is in the closure of $V^+$, $V^0$ is also in the closure of $\Phi_{\mca E *} U(Y)$. 
Then we claim $V^- \cap \Phi_{\mca E * } U(Y) = \emptyset$. 
In fact, if $V^- \cap \Phi_{\mca E * }U(Y) \neq \emptyset$ then $\mca E_y$ is a $\mu$-stable locally free sheaf for all $y \in Y$ by Proposition \ref{B.2}. 
Moreover $V^0$ is in $\Phi_{\mca E* }U(Y) $ by Proposition \ref{A.6}. 
This contradicts $V^0 \cap \Phi_{\mca E *}U(Y) = \emptyset$. 
Hence we see $V^- \cap \Phi_{\mca E*}U(Y) = \emptyset$. 
Thus $V^0$ is contained in the boundary $\partial (\Phi_{\mca E*}U(Y) )$. 
Moreover any $\sigma \in V^0$ is general in $\partial (\Phi_{\mca E*}U(Y) )$ since there are no walls in $V(X)$. 

Take a stability condition $\sigma_0 \in V^0$. 
Recall that the Picard number of $X$ is $1$. 
Since $Y$ is a Fourier-Mukai partner of $X$, the Picard number of $Y$ is also $1$. 
Since there is no $(-2)$-curve in $Y$, $\mca O_y$ is properly $\Phi_{\mca E*}^{-1}\sigma_0$-semistable for all $y \in Y$ by Theorem \ref{wall}. 
Hence $\mca E_y$ is not $\sigma_0$-stable but $\sigma _0$-semistable. 
Moreover we see that $\mca E_y$ is not a locally free sheaf by Propositions \ref{B.1} and \ref{B.2}. 
Hence we have the following distinguished triangle by taking the double dual of $\mca E_y$:
\[
\begin{CD}
S_y[-1] @>>> \mca E_y @>>> \mca E_y^{\vee \vee}@>>> S_y ,
\end{CD}
\]
where $S_y = \mca E_y^{\vee \vee }/\mca E_y$. 
By Lemma \ref{hom1}, we see 
\[
\hom_X^1(S_y, S_y) =2\mbox{ and }\hom_X^1(\mca E_y^{\vee \vee}, \mca E_y^{\vee \vee})=0.
\] 
Thus there is a closed point $x \in X$ such that $S_y=\mca O_x$.  
Since $\sigma _0$ is in $V(X)$, $\mca O_x$ is a $\sigma_0$-stable factor of $\mca E_y$. 
By Theorem \ref{wall}, $\mca E_y^{\vee \vee}$ is a direct sum of a spherical object $S$. 
Since $\mca E^{\vee \vee}_y$ is a locally free sheaf, $S$ is an also locally free sheaf with $\mu_{L}(S) = \mu_{L}(\mca E_y^{\vee \vee})$. 
Thus we can put $\mca E_y^{\vee \vee} = S^{\oplus \ell}$. 

Since $v(\mca E_y) = v(\mca E_y^{\vee \vee})-0\+ 0\+ 1$, we have
\begin{equation}
0 = v(\mca E_y)^2 = v(\mca E_y^{\vee \vee})^2-2 \< v(\mca E_y^{\vee \vee}), v(\mca O_x)\>.  \label{koutousiki}
\end{equation}
Furthermore we have $v(\mca E_y^{\vee \vee})^2 = -2 \ell ^2$ and 
\[
\< v(\mca E_y^{\vee \vee}), v(\mca O_x) \> = -\rank \mca E_y^{\vee\vee} = -\rank \mca E_y = -r.
\] 
Thus we have 
\[
2 \ell ^2 = 2 r. 
\]
Hence if $r$ is not a square number then we have $V^0 \cap \Phi _{\mca E *} U(Y)\neq \emptyset $. 
Thus $\mca E_y$ is a $\mu$-stable locally free sheaf for all $y \in Y$ by Proposition \ref{A.6}. 
This gives the proof of the first assertion $(1)$. 

Suppose that $\rank \mca E_y$ is a square number. 
Then a JH filtration of $\mca E_y$ is given by the following triangle:
\[
\begin{CD}
\mca O_x[-1] @>>> \mca E_y @>>> S^{\oplus r} @>>> \mca O_x.
\end{CD}
\]
Since $\mca O_x[-1]$ is the unique stable factor of $\mca E_y$ with an isotropic Mukai vector, one of the following two cases will occur by Theorem \ref{wall} and by the uniqueness of stable factors up to permutations:
\begin{itemize}
\item[(i)] For any $y \in Y$, there is a closed point $x \in X$ such that $\Phi_{\mca E}\circ T_B (\mca O_y) = \mca O_x[-1]$ where $B$ is a spherical locally free sheaf on $Y$ and $T_B$ is the spherical twist by $B$. 
\item[(ii)] For any $y \in Y$, there is a closed point $x \in X$ such that $\Phi_{\mca E}\circ T_B^{-1} (\mca O_y) = \mca O_x[-1]$ where $B$ is a spherical locally free sheaf on $Y$. 
\end{itemize}
We remark that $B$ does not depend on $y$ by Theorem \ref{wall}. 

Assume that the first case (i) occurs. 
Then, as is well-known, there is a line bundle $M$ on $X$ and an isomorphism $f: Y \to X$ such that $\Phi _{\mca E} \circ T_B(-) = M \otimes f_*(-) [-1]$\footnote{For instance see \cite[Corollary 5.23]{Huy2}}. 
Thus we have 
\begin{equation}
\Phi_{\mca E}(\mca O_y) = M \otimes f_* (T_B^{-1}(\mca O_y))[-1]. \label{kome3}
\end{equation}
Then the right hand side of (\ref{kome3}) is properly complex and the left hand side is a sheaf. 
This is contradiction. 
Hence the second case (ii) should occur. 
Then $\Phi_{\mca E}(-)$ is given by $M \otimes f_* (T_B(-))[-1]$. 
This gives the proof of the second assertion $(2)$. 
\end{proof}

\begin{ex}\label{B.5}
Let $(X,L)$ be a generic K3 and let $E$ be a Gieseker stable torsion free sheaf with $v(E) = r\+ n L \+ s$. 
Since $\mr{NS}(X)= \bb Z L$, $E$ is $\mu$-stable if $\gcd \{ r,n \}=1$ by \cite[Lemma 1.2.14]{HL}. 
Then $E$ is a $\mu$-stable locally free sheaf by Proposition \ref{B.1}. 
Moreover if $\gcd \{ r, nL^2,s \}=1$ then the moduli space containing $E$ is a fine moduli space. 

Let $(X,L)$ be a generic K3 with $L^2=6$. 
Take $v \in \mca N(X)$ as $v= 12\+ 10 L \+ 25$. 
Then by \cite[Corollary 4.6.7]{HL} the moduli space $M_L(v)$ of Gieseker stable torsion free sheaves with Mukai vector $v$ is the fine moduli space since $\gcd  \{ 12, 10 L^2, 25 \}=1$. 
By Theorem \ref{B.4}, $M_L(v)$ is the moduli space of $\mu_L$-stable locally free sheaves, although $\gcd \{12,10  \}=2$. 
\end{ex}

\section{Second application}\label{C}

The goal of this section is to generalize \cite[Theorem 1.1]{Kaw} to arbitrary projective K3 surfaces. 

In \cite{Kaw} the author describes a picture of $(T_L)_*U(X) \cap V(X)$ by using \cite[Theorem 1.2]{Kaw} where $T_L$ is a spherical twist by an ample line bundle $L$. 
Instead of the theorem we use Lemma \ref{C.1} (below). 
Before we state the lemma we prepare the notations. 
Let $X$ be a projective K3 surface and take an ample line bundle $L$. 
For the line bundle $L$ we define the subset $V_L^{>0}$ of $V(X)$ by
\[
V_L^{>0} := \{ \sigma_{(\beta, \omega)} \in V(X)_L |  
(L^2 -\beta L) \leq \frac{\omega ^2}{2}  \}.
\]
The following lemma is essentially contained in Theorem \ref{A.6}. 
However we write down the lemma to make it much easier to use Theorem \ref{A.6}. 

\begin{lem}\label{C.1}
Notations are being as above. 
The set $V_L^{>0}$ is contained in $T_{L*}U(X)$. 
\end{lem}

\begin{proof}
Recall that $T_L(\mca O_x)= L \otimes \mca I_x[1]$ where $\mca I_x$ is the kernel of the evaluation map $\mca O_X \to \mca O_x$.  
If $\sigma $ is in $V_L^{>0}$ then $L \otimes \mca I_x$ is $\sigma$-stable  for all $x \in X$ by Theorem \ref{A.4}. 
Furthermore the phase of $L \otimes \mca I_x$ does not depend on $x \in X$. 
Thus we have proved the assertion. 
\end{proof}

The following lemma is also used in \cite{Kaw}. 
By using Lemma \ref{C.2}, we can see $\Phi(\mca O_y)$ is a sheaf up to shifts if an equivalence $\Phi:D(Y) \to D(X)$ satisfies the condition $\Phi_*U(Y) =U(X)$.

\begin{lem}\label{C.2}(\cite[Proposition 14.2]{Bri2}, \cite[Proposition 6.4]{Tod})
Let $X$ be a projective K3 surface, $E$ in $D(X)$ and $\sigma _{(\beta, \omega)}=(\mca A,Z ) \in V(X)$. 
We put $v(E) = r_E \+ \delta _E \+ s_E$. 

$(1)$ Assume that $r_E >0$ and $E \in \mca A$. 
If there exists a positive real number $\ell _0$ such that $E$ is $\sigma _{(\beta, \ell \omega )}$-stable for all $\ell > \ell _0$, 
then $E$ is a torsion free sheaf and is $(\beta , \omega)$-twisted stable. 

$(2)$  Assume that $r_E =0$ and $E \in \mca A$. 
If there exists a positive real number $\ell _0$ such that $E$ is $\sigma _{(\beta, \ell \omega )}$-stable for all $\ell > \ell _0$, 
then $E$ is a pure torsion sheaf. 
\end{lem}

In \cite{Kaw} the author proves that some spherical twists send sheaves to complexes in some special cases. In the following Lemma we generalize this result to arbitrary projective K3 surfaces. 

\begin{lem}\label{C.3}
Let $X$ be a projective K3 surface and let $E$ and $A$ be coherent sheaves with positive rank. We assume that $v(E)^2=0$ and $v(A)^2=-2$ and 
put $v(E) = r_E \+ \delta _E \+ s_E$ and $v(A)= r_A \+ \delta _A \+ s_A$. 

$(1)$ If $(\frac{\delta _E}{r_E}-\frac{\delta_A}{r_A})^2 \geq 0$ then $\chi (A,E)>0$. 

$(2)$ In addition to $1$, assume that $A$ is spherical and $\hom^0_X(A,E)=0$. 
Then the spherical twist $T_A(E)$ of $E$ by $A$ is a complex. 
In particular the $0$-th and $1$-st cohomologies survive. 
\end{lem}

\begin{proof}
We first show the first assertion. 
Since $r_E$ and $r_A$ are positive, it is enough to show that $\frac{\chi(A,E)}{r_A r_E}$ is positive. 
We have 
\begin{eqnarray*}
\frac{\chi(A,E)}{r_A r_E} &=& - \< 1 \+ \frac{\delta _A}{r_A} \+ \frac{s_A }{r_A}  , 1\+ \frac{\delta _E}{r_E}\+ \frac{s_E}{r_E}\> \\
						&=& \frac{s_A}{r_A} + \frac{s_E }{r_E} - \frac{\delta _A \delta _E}{r_A r_E}.
\end{eqnarray*}
Since $v(A)^2 =-2$ and $v(E)^2=0$ we have 
\[
\frac{s_A}{r_A} = \frac{1}{2}\frac{\delta _A^2}{r_A^2} + \frac{1}{r_A^2} 
\mbox{ and }
\frac{s_E}{r_E} = \frac{1}{2}\frac{\delta _E^2}{r_E ^2}. 
\]
Thus we have
\begin{eqnarray*}
\frac{\chi (A,E)}{r_A r_E} &=& \frac{1}{2}\frac{\delta _A ^2}{r_A^2} + \frac{1}{r_A^2} + \frac{1}{2} \frac{\delta _E^2}{r_E^2} -\frac{\delta _A \delta _E }{r_A r_E} \\
						&=& \frac{1}{2}\Bigl( \frac{\delta _A}{r_A} -\frac{\delta _E}{r_E} \Bigr)^2 + \frac{1 }{r_A^2} >0.
\end{eqnarray*}
Thus we have proved the first assertion. 

We show the second assertion. 
By the assumption and $(1)$ of Lemma \ref{C.3} we have 
$\chi (A,E) = -\hom_X^1(A,E) + \hom_X^2(A,E) >0$. 
Hence $\hom_X^2(A,E)$ is not $0$. 
By the computing of the $i$-th cohomology $H^i$ of $T_A(E)$, we can prove the assertion. In fact we have the following exact sequence of sheaves: 
\[
\begin{CD}
@. \Hom_X^0(A,E) \otimes A@>>> E @>>> H^0 \\
@>>>\Hom^1_X(A,E) \otimes A @>>> 0 @>>> H^1\\
@>>>\Hom^2_X(A,E) \otimes A @>>> 0 . 
\end{CD}
\]
Since $\hom_X^2(A,E)$ is not $0$, we see $H^1 \neq 0$. 
Since $\hom_X^0(A,E) $ is $0$, the sheaf $H^0 $ contains $E$. 
Thus $H^0$ is not $0$. 
\end{proof}

For an equivalence $\Phi$ satisfying the condition $\Phi_{\mca E*}U(Y)=U(X)$ and for a closed point $y \in Y$, it is enough to prove $\Phi(\mca O_y) =\mca O_x[n]$ for some $x \in X$ and $n \in \bb Z$. 
By Lemma \ref{C.2}, if $\Phi_*U(Y) = U(X)$ then $\Phi(\mca O_y)$ should be a sheaf up to shifts. 
Thus we have to exclude the case $\Phi(\mca O_y)$ is a torsion free sheaf $F$ or pure torsion sheaf $T$ with $\dim \mr{Supp}(T)=1$ (up to shifts). 
If the Picard number of $X$ is one 
then it is not necessary to consider the case $\Phi(\mca O_y)=T$ with $\dim \mr{Supp}(T)=1$ since $v(\Phi(\mca O_y))^2=0$.  
We need the following lemma to exclude the case $\Phi(\mca O_y)=T$ with $\dim \mr{Supp}(T)=1$. 

\begin{lem}\label{C.4}
Let $X$ be a projective K3 surface, $E$ a pure torsion sheaf with $\dim \mr{Supp}(E)=1$ and $L$ a line bundle on $X$. 
If $\chi (L,E) < 0$ then the spherical twist $T_L(E)$ of $E$ is a sheaf containing a torsion sheaf or a properly complex. 
In particular $T_L(E)$ is not a torsion free sheaf. 
\end{lem}

\begin{proof}
Since $E$ is torsion and $L$ is torsion free we have $\hom^2_X(L,E)=0$ by the Serre duality. 
Thus  we have $\hom^1_X(L,E) \neq 0$ by $\chi (L,E ) < 0$. 
We can compute the $i$-th cohomology $H^i$ of $T_L(E)$ in the following way: 
\[
\begin{CD}
@. @. 0 @>>> H^{-1} \\
@>>> \Hom^0_X(L,E )\otimes L @>>> E @>>> H^0 \\
@>>> \Hom^1_X(L,E)\otimes L @>>> 0.
\end{CD}
\]
Since $\hom^1_X(L,E) \neq 0$ we see $H^0 \neq 0$. 

Suppose that $\Hom_X^0(L,E)=0$. Then $H^{-1}=0$. We can easily see $H^i=0$ if $i \neq 0$. 
Hence $T_L(E)$ is a sheaf containing the torsion sheaf $E$. 

Suppose that $\Hom_X^0(L,E)\neq 0$. Since $E$ is torsion, $H^{-1}$ is not $0$. 
Thus $T_L(E)$ is a comlex. 
\end{proof}

In Proposition \ref{C.5} and Corollary \ref{C.6}, we generalize \cite[Theorem 6.6]{Kaw}.

\begin{prop}\label{C.5}
Let $X$ be a projective K3 surface and $E$ in $D(X)$ with $v(E)^2=0$. 
We put $v(E) = r_E \+ \delta _E \+ s_E$. 

$(1)$ Suppose that $r_E\neq 0$. Then there is a $\sigma \in V(X)$ such that $E$ is not $\sigma $-stable. 

$(2)$ Suppose that $r_E=0$ and $E $ is $\sigma $-stable for all $\sigma \in V(X)$. Then $E$ is $\mca O_x[n]$ for some closed points $x \in X$ and $n \in \bb Z$. 
\end{prop}

\begin{proof}
Let us prove the first assertion $(1)$. 
Suppose to the contrary that $E$ is $\sigma $-stable for all $\sigma \in V(X)$. 
Since $r_E\neq 0 $ we can assume $r_E >0$ by a shift if necessary. 
We choose a stability condition $\sigma_{(\beta_0 , \omega_0)}=(\mca A_0, Z_0) \in V(X)$ so that 
$\frac{\delta _E\omega _0}{r_E} > \beta_0 \omega _0$ and $\omega _0$ is an integral class. 
Since $\frac{\delta _E\omega _0}{r_E} > \beta_0 \omega _0$ the imaginary part $\mf{Im} Z_0(E)$ of $Z_0(E) $ is positive.  
Hence there is an even integer $2m$ such that $E[2m]$ is in $\mca A_0$. 
Thus we rewrite $E$ instead of $E[2m]$. Note that $E$ is in $\mca A_0$ and $r_E$ is positive. 

We consider the following one parameter family of stability conditions 
\[
\{ \sigma _{\ell} := \sigma _{(\beta _0, \ell \omega _0)} \in V(X) | \ell \in \bb R_{>>0} \}. 
\]
We put $\sigma _{\ell} =(\mca A_{\ell },Z_{\ell})$. 
By (1) of Lemma \ref{C.2}, $E$ is a $(\beta _0 , \omega _0)$-twisted stable torsion free sheaf. 

We choose am ample line bundle $L$ satisfying the following condition:
\begin{enumerate}
\item $c_1(L) = n \omega _0$ where $n$ is a positive integer. 
\item $\mu_{\omega _0}(L) > \mu_{\omega _0}(E)$. 
\item $(\frac{\delta _E}{r_E}-L)^2 >0$. 
\item $r_E -\chi (L,E) < 0$. 
\end{enumerate}
This choice is possible if we take a sufficiently large $n$. 
Since $E$ is twisted stable, $E$ is $\mu$-semistable with respect to $\omega _0$. Thus $\hom_X^0(L,E) =0$ by the second condition for $L$.  
Hence $T_L(E)$ is a complex by Lemma \ref{C.3}. 
In particular the $0$-th and $1$-st cohomologies survive. 

Now we put $E' = T_L(E)[1]$ and $v(E') = r' \+ \delta ' \+ s'$. 
Since $r' = \chi (L,E) -r_E$, $r'$ is positive. 
We choose a divisor $\beta $ so that 
\[
\beta = b L \ (b\in \bb R) \mbox{ and }\beta \omega _0 < \min\{L \omega _0, \frac{\delta ' \omega _0}{r'}  \}. 
\]
We consider the following family of stability conditions:
\[
\{ \sigma _{y}:= \sigma _{(\beta , y L_0)} \in V_L^{>0} | L_0^2 -\beta L_0 \leq   \frac{(yL_0)^2}{2}  \}. 
\]
We put $\sigma _y =(Z_y, \mca P_y)$. 
By Lemma \ref{C.1}, a stability condition $\sigma _y$ is in $(T_L)_*U(X)$. 
Since $E$ is $\tau$-stable for all $\tau \in U(X)$, the object $E'$ is $(T_L)_*\tau$-stable. 
Thus $E'$ is $\sigma _y$-stable since $\sigma_y $ is in $T_{L*}U(X)$. 
By the choice of $\beta$ we have $\mf{Im}Z_y(E')>0$. 
Hence $E'$ should be a torsion free sheaf up to shifts by $(1)$ of Lemma \ref{C.2}. 
This contradicts the fact that two cohomologies of $E'$ survive. 

Let us prove the second assertion $(2)$. 
We choose an arbitrary stability condition $\sigma _{(\beta _0, \omega _0)}=(\mca A_0, Z_0) \in V(X)$ and fix it. 
Since $E$ is $\sigma _{(\beta_0, \omega _0)}$-stable we can assume that $E$ is in $\mca A_0$ by shifts if necessary. 
By taking a limit $\omega _0 \to \infty $ we see that $E$ is a pure torsion sheaf by $(2)$ of Lemma \ref{C.2}. 

We shall show $\delta _E =0$. Suppose to the contrary that $\delta _E \neq 0$. 
Then $\delta _E L$ is positive for any ample line bundle $L$. 
Thus there is a sufficiently ample line bundle $L_0$ such that $\chi (L_0, E) < 0$. 
Here we put $v(T_{L_0}(E)) = r \+ \delta \+ s$. 
Since $r= -\chi (L_0, E)$, we see $r> 0$.  
Similarly to $1$ we consider the following family of stability conditions
\[
\{ \sigma _y := \sigma_{(0, y L_0)}=(\mca A_y, Z_y) | L_0^2 \leq \frac{(y L_0)^2}{2} \}. 
\]
Since $\mu_{L_0}(L_0) = L_0^2 >0$, $\sigma _y$ is in $(T_{L_0})_*U(X)$ by Lemma \ref{C.1}. 
Moreover we have 
\[
\frac{\delta L_0}{r} = \frac{\delta _E -\chi (L_0,E)L_0}{r}L_0>0. 
\]
Thus $\mf{Im}Z_y(T_{L_0}(E)) >0$. 
Hence we can assume that $T_{L_0}(E)$ is in $\mca A_y$ up to even shifts. 
By $(1)$ of Lemma \ref{C.2} $T_{L_0}(E)$ should be a torsion free sheaf. 
This contradicts Lemma \ref{C.4}. Thus we have $\delta _E =0$. 

Since $\delta _E=0$, $E$ is a pure torsion sheaf with $\dim \mr{Supp}(E) =0$. 
Since $E$ is $\sigma $-stable we have $\hom^0_X(E,E) =1$. 
Thus $E$ is a length $1$ torsion sheaf up to shifts. 
We have proved the assertions. 
\end{proof}

\begin{cor}\label{C.6}
Let $X$ be a projective K3 surface and $E$ in $D(X)$ with $v(E)^2=0$. 
If $E$ is $\sigma $-stable for all $\sigma \in V(X)$ then 
$E$ is $\mca O_x[n]$ for some $x \in X$ and $n \in \bb Z$. 
\end{cor}

\begin{proof}
We put $v(E) = r_E \+ \delta_E \+ s_E$. If $r_E \neq 0$ then this contradicts $(1)$ of Proposition \ref{C.5}. Hence $r_E=0$. 
The assertion follows from $(2)$ of Proposition \ref{C.5}. 
\end{proof}

\begin{thm}\label{C.7}
Let $X$ and $Y$ be projective K3 surfaces and $\Phi : D(Y) \to D(X)$ an equivalence. 
If $\Phi _* U(Y) = U(X)$ then $\Phi$ can be written by 
\[
\Phi (-) = L\otimes f_*(-) [n]
\]
where $L$ is a line bundle on $X$, $f$ is an isomorphism $f:Y\to X$ and $n \in \bb Z$. 
\end{thm}

\begin{proof}
Take an element $\sigma \in \Stab (X)$. 
By the definition of $\tilde {GL}^+ (2,\bb R)$ action we see that an object $E$ is $\sigma$-stable if and only if $E$ is $\sigma \tilde g$-stable for all $\tilde g \in \tilde {GL}^+(2, \bb R)$. 
Hence if $\Phi _*U(Y) = U(X)$ then $\Phi (\mca O_y)$ is written by $\mca O_x[n]$ for some $x \in X$ and $n \in \bb Z$ by Corollary \ref{C.6}. 
Then the assertion follows from \cite[Corollary 5.23]{Huy2}. 
\end{proof}

Then we immediately obtain the following corollary. 

\begin{cor}\label{9}
We put 
\[
\Aut (D(X), U(X)):= \{ \Phi \in \Aut (D(X)) | \Phi _*U(X) =U(X) \}. 
\]
Then $\Aut (D(X), U(X)) = (\Aut (X)\ltimes \mr{Pic}(X) )\times \bb Z [1]$. 
\end{cor}

\noindent
{Graduate School of Mathematics, Osaka University, Osaka 563-0043, Japan};
{e-mail: kawatani@cr.math.sci.osaka-u.ac.jp}


\begin{thebibliography}{ABCD}

\bibitem[1]{Bri}
{T. Bridgeland}, 
{\em Stability conditions on triangulated categories},
{Ann. of Math.} \textbf{166} (2007), 317--345.

\bibitem[2]{Bri2}
{T. Bridgeland},
{\em Stability conditions on K3 surfaces},
{Duke Math. J.} \textbf{141} (2008), 241--291. 



\bibitem[3]{HL}
{D. Huybrechts} and {M. Lehn},
{\em The geometry of moduli spaces of sheaves},
{Aspects of Mathematics}, 1997.

\bibitem[4]{HMS}
{D. Huybrechts}, {E. Macri} and {P. Stellari},
{\em Stability conditions for generic K3 categories},
{Compositio Math.} \textbf{144} (2008),134--162.

\bibitem[5]{Huy}
{D. Huybrechts}, 
{\em Derived and abelian equivalence of K3 surfaces},
{J, Algebraic Geom.} \textbf{17} (2008), 357--400. 

\bibitem[6]{Huy2}
{D. Huybrechts},
{\em Fourier-Mukai transformations in Algebraic Geometry},
{Oxford science publications}, 2006.

\bibitem[7]{MW}
{K. Matsuki and R. Wentworth}, 
{\em Mumford-Thaddeus principles on the moduli space of vector bundles on an algebraic surface}, 
{Internat. J. Math.} \textbf{8} (1997), 97--148. 

\bibitem[8]{Kaw}
{K. Kawatani}, 
{\em Stability conditions and $\mu$-stable sheaves on K3 surfaces with Picard number one}, 
{arXiv:math. 1005.3877}, {to appear in Osaka J. Math. }

%

\bibitem[9]{Tod}
{Y. Toda}, 
{\em Moduli stacks and invariants of semistable objects on K3 surfaces}, 
{Adv. Math.} \textbf{217} (2008), 2736--2781. 
\end{thebibliography}
\end{document}